\documentclass[oneside,english]{amsart}
\usepackage[T1]{fontenc}
\usepackage[latin9]{inputenc}
\usepackage{amsthm}
\usepackage{amstext}
\usepackage{amssymb}

\numberwithin{equation}{section}
\numberwithin{figure}{section}
\theoremstyle{plain}
\newtheorem{thm}{Theorem}
\theoremstyle{plain}
\newtheorem*{prop*}{Proposition}
\theoremstyle{definition}
\newtheorem{defn}[thm]{Definition}
\theoremstyle{remark}
\newtheorem*{rem*}{Remark}
\theoremstyle{plain}
\newtheorem{lem}[thm]{Lemma}
\theoremstyle{remark}
\newtheorem*{claim*}{Claim}
\theoremstyle{plain}
\newtheorem{prop}[thm]{Proposition}
\theoremstyle{plain}
\newtheorem{cor}[thm]{Corollary}
\theoremstyle{definition}
\newtheorem{example}[thm]{Example}

\title{One Dimensional t.t.t Structures}

\author{Daniel Lowengrub}

\address{Department of Mathematics\\ The Hebrew University of Jerusalem\\
  Jerusalem, 91904, Israel}

\email{lowdanie@gmail.com}

\begin{document}

\maketitle

\section{Introduction}

The notion of a first order topological structure was introduced by
Pillay \cite{P1} as a generalization of the notion of an o-minimal
structure. The idea is to provide a general framework in which model
theory can be used to analyze a topological structure whose topology
isn't necessarily induced by a definable order. In the o-minimal case,
the topology is generated from a basis where each basis set can be
defined by substituting the variables $y_{1}$ and $y_{2}$ by suitable
parameters in the following formula\[
\phi(x,y_{1},y_{2})=y_{1}<x<y_{2}\]

A first order topological structure generalizes this to the case where
$\phi$ is some arbitrary formula with more than one variable.

Pillay also introduced the notion of topologically totally transcendental
(t.t.t) structures which have the additional property that any definable
set has a finite number of connected components. For example, by definition
o-minimal structures are t.t.t.

In the previously mentioned paper, Pillay proved that one dimensional
t.t.t structures have some characteristics in common with o-minimal
structures such as the exchange property. Furthermore, he showed that
if the topology of a one dimensional t.t.t structure is induced by
a definable dense linear ordering then the structure is o-minimal.

In this paper we'll focus on $\omega$-saturated one dimensional t.t.t
structures and prove that under a few additional topological assumptions,
such structures are composed of o-minimal components in a relatively
simple manner.

Our main result which will be proved in section \ref{sec:Structures-With-Splitting}
will be to show that if we assume that removing any point from the
structure splits it into at least two connected components, then the
structure must be a one dimensional simplex of a finite number of
o-minimal structures:
\begin{thm}
  Let $M$ be a 1-dimensional connected $\omega$-saturated t.t.t structure
  such that for each point $x\in M$, $M\backslash\{x\}$ has at least
  two definably connected components. Then there exists a finite set
  $X\subset M$ such that each of the finite number of connected components
  of $M\backslash X$ are o-minimal.\label{thm:intro-splitting}
\end{thm}
In section \ref{sec:Structures-Without-Splitting} we'll analyze the
case where removing a point doesn't necessarily split the structure,
and will find two alternative topological properties which guarantee
that the structure is locally o-minimal. This is done by showing that
basis sets which are small enough can be split by removing a point.
\begin{thm}
  Let $M$ be a 1-dimensional $\omega$-saturated t.t.t structure
  such that one of the following holds:\label{thm:intro-no-splitting}
  \begin{enumerate}
  \item There exist a definable continuous function $F:M^{2}\rightarrow M$
    and a point $a\in M$ such that for each $x\in M$, $F(x,x)=a$
    and $F(x,\cdot)$ is injective.
  \item For every basis set $U$, $\vert bd(U)\vert=2$.
  \end{enumerate}
  Then for all but a finite number of points, for every point $x\in M$
  there's a basis set $U$ containing $x$ such that $U$ is o-minimal.
\end{thm}
An immediate corollary of part 1 of theorem \ref{thm:intro-no-splitting} is
that if an $\omega$-saturated one dimensional t.t.t structure admits
a topological group structure then it is locally o-minimal.

Towards the end of the section we'll prove a version of the monotonicity
theorem for locally o-minimal structures. This shows that locally
o-minimal structures share many characteristics with standard o-minimal
structures.

Theorem \ref{thm:intro-splitting} illustrates that the defining characteristic
of o-minimal structures within the general setting of $\omega$-saturated
one dimensional t.t.t structures isn't existence of the order itself,
but rather the ability to disconnect the structure by removing a point.

Theorem \ref{thm:intro-no-splitting} shows that even in the case
where an $\omega$-saturated one dimensional t.t.t structure isn't
o-minimal, it will at least be o-minimal on a local scale provided
that it has a rudimentary internal structure.

An important step in proving the theorems above will be to show that
the relation $a\sim_{x}b$, which says that $a$ and $b$ are in the
same connected component of $M\backslash\{x\}$, is definable.
\begin{prop*}
  Let $M$ be a 1-dimensional $\omega$-saturated t.t.t structure such
  that for each point $x\in M$, $M\backslash\{x\}$ has more than one
  connected component. Then the relation $a\sim_{x}b\subset M^{3}$
  is definable.
\end{prop*}
In section \ref{sec:Connected-Components-in} we'll prove that that
the number of connected components in a definable family is uniformly
bounded.
\begin{prop*}
  Let $(M,\phi)$ be a $1$-dimensional $\omega$-saturated t.t.t structure
  and \[
  \alpha(x,y_{1},\dots,y_{l})\in L\]
  . Then there exists a constant $C\in\mathbb{N}$ such that for every
  $l$-tuple $c_{1},\dots,c_{l}\in M$, $\alpha(c_{1},\dots,c_{l})$
  has less than $C$ connected components.
\end{prop*}
This in turn will allow us to prove that elementary extensions of
such structures are t.t.t as well.

\section{Preliminaries }

In this section we review some of the notions, definitions and results
from Pillay \cite{P1} which will be used heavily throughout the following
sections.
\begin{defn}
  Let $M$ be a two sorted $L$ structure with sorts $M_{t}$ and $M_{b}$
  and let $\phi(x,y_{1},\dots,y_{k})$ be an $L$ formula such that
  $\{\phi^{M_{t}}(x,\bar{a})\vert\bar{a}\in M_{b}^{k}\}$ is a basis
  for a topology on $M_{t}$. Then the pair $(M,\phi)$ will be called
  a \emph{first order topological structure}. When we talk about the
  topology of $M_{t}$ we mean the one generated by the basis described
  above.\end{defn}
\begin{rem*}
  In Pillay's paper, first order topological structures were defined
  on a one-sorted structure where each element can be both a parameter
  for a basis set, and a point in the topological space. However, in
  practice this double meaning isn't needed, so we're using the two
  sorted definition both for clarity and in order to slightly strengthen
  some of the theorems.
\end{rem*}
In addition, we consider the following condition on a \emph{ }first
order topological structure $M$:

(A) Every definable set $X\subset M_{t}$ is a boolean combination
of definable open subsets.

In this paper we assume that $M_{t}$ is Hausdorff and $(M,\phi)$
is a first order topological structure satisfying (A).

The following topological result is also helpful in this context and
was proved by Robinson \cite[4.2]{R}.
\begin{lem}
  Let $V$ be a topological space, and $W\subset V$ a non-empty subset.
  Let $A\subset V$ be a boolean combination of open subsets of $V$
  and let $B=V\backslash A$. Then either $W\cap A$ or $W\cap B$ has
  an interior with respect to the induced topology on $W$.\label{lem:set-or-comp-int}\end{lem}
\begin{defn}
  Let $M$ be a first order topological structure satisfying (A) and let
  $X\subset M_{t}$ be a closed definable subset of $M_{t}$. The ordinal
  valued $D_{M}(X)$ is defined by:
  \begin{enumerate}
  \item If $X\neq\emptyset$ then $D_{M}(X)\geq0$.
  \item If $\delta$ is a limit ordinal and $D_{M}(X)\geq\alpha$ for all
    $\alpha<\delta$ then $D_{M}(X)\geq\delta$.
  \item If there's a closed definable $Y\subset M_{t}$ such that $Y\subset X$,
    $Y$ has no interior in $X$ and $D_{M}(Y)\geq\alpha$ then $D_{M}(X)\geq\alpha+1$.
  \end{enumerate}
\end{defn}
\begin{rem*}
  We'll write $D_{M}(X)=\alpha$ if $D_{M}(X)\geq\alpha$ and $D_{M}(X)\ngeq\alpha+1$.
  We'll write $D_{M}(X)=\infty$ if $D_{M}(X)\geq\alpha$ for all $\alpha$.\end{rem*}
\begin{defn}
  We say that $M$ \emph{has dimension }if $D_{M}(X)\neq\infty$ for
  all closed definable subsets $X\subset M_{t}$.
\end{defn}
In addition, we define the number of definable connected components
for definable subsets of our topology:
\begin{defn}
  Let $X\subset M_{t}$ be definable. Then $d_{M}(X)$ is the maximum number
  $d<\omega$ such that there are disjoint definable clopen sets $X_{1},\dots,X_{d}\subset X$
  with $X=\cup_{i=1}^{d}X_{i}$ , and $\infty$ if no such $d$ exists.\end{defn}
\begin{rem*}
  Throughout the paper, when we say {}``connected'' we always mean
  {}``definably connected''.
\end{rem*}
And now for the main definition:
\begin{defn}
  We say that $M$ is \emph{topologically totally transcendental (t.t.t)}
  if $M$ is a first order topological structure satisfying (A) with
  dimension such that for every definable set $X\subset M_{t}$, $d_{M}(X)<\infty$.
  We say that a theory $T$ is t.t.t is every model of $T$ is t.t.t.
\end{defn}
The following lemma was proved by Pillay \cite[6.6]{P1} and plays
a key role in most of the proofs in this paper.
\begin{lem}
  Let $M$ be a 1-dimensional t.t.t structure. Then:\label{lem:Let--be}
  \begin{enumerate}
  \item For any closed and definable $X\subset M_{t}$, $D(X)=0$ iff $X$
    is finite.
  \item The set of isolated points of $M_{t}$ is finite.
  \item For any definable $X\subset M_{t}$ there are pairwise disjoint definably
    connected definable open subsets $X_{1},\dots,X_{m}\subset M_{t}$
    and a finite set $Y\subset M_{t}$ such that $X=(\cup_{i=1}^{m}X_{i})\cup Y$.
  \item For any definable $X\subset M_{t}$, the set of boundary points of
    $X$ is finite.
  \end{enumerate}
\end{lem}
\begin{rem*}
  One consequence of part $3$ of lemma \ref{lem:Let--be} which will
  be used many times below is that if a set $A\subset M_{t}$ is definable
  then the statement \emph{{}``$A$ is infinite'' }is expressible
  in first order logic as it's equivalent to the statement \emph{{}``$A$
    has no interior''}.
\end{rem*}

\section{Connected Components in Definable Families\label{sec:Connected-Components-in}}

In this section we'll show that the number of connected components
is uniformly bounded over a definable family. This is used to prove
that in 1-dimensional $\omega$-saturated structures, the property
of being t.t.t is preserved under elementary equivalence.
\begin{lem}
  Let $(M,\phi)$ be a $1$-dimensional $\omega$-saturated t.t.t structure.
  Then there exists a number $K\in\mathbb{N}$ such that for each point $b\in M_{b}$,
  $\vert bd(\phi^{M_{t}}(b)\vert\leq K$.\label{lem:bounded-boundary}\end{lem}
\begin{proof}
  For each $b\in M_{b}$, $\vert bd(\phi^{M_{t}}(b)\vert$ is finite.
  The lemma then follows from the fact that $M$ is $\omega$-saturated.\end{proof}
\begin{lem}
  Let $(M,\phi)$ be a definably connected $1$-dimensional $\omega$-saturated
  t.t.t structure, $K\in\mathbb{N}$ a number such that for each point $b\in M_{b}$
  we have $\vert bd(\phi^{M_{t}}(b)\vert\leq K$, and
  $X\subset M_{t}$ a definable subset such that $bd(X)=n$. Then $d_{M}(X)\leq n\cdot K$.\label{lem:bounded-connected}\end{lem}
\begin{proof}
  Let $N=d_{M}(X)$ and let $\{Y_{1},\dots,Y_{N}\}$ be pairwise disjoint
  clopen (in $X$) subsets of $X$ such that $X=\cup_{i=1}^{N}Y_{i}$.
  In addition, we denote the elements of $bd(X)$ by $bd(X)=\{a_{1},\dots,a_{n}\}$. 

  By the Hausdorffness of $M_{t}$, we can find basis sets $\{U_{1},\dots,U_{n}\}$
  such that for all $1\leq i\leq n$:
  \begin{enumerate}
  \item $a_{i}\in U_{i}$
  \item For all $1\leq j\leq N$, if $Y_{j}\neq\{a_{i}\}$ then $Y_{j}\backslash U_{i}\neq\emptyset$.\end{enumerate}
  \begin{claim*}
    For all $1\leq j\leq N$, if $Y_{j}$ isn't a point then there exists
    an $1\leq i\leq n$ such that $a_{i}\in\bar{Y_{i}}$ and $bd(U_{i})\cap Y_{j}\neq\emptyset$.\end{claim*}
  \begin{proof}
    Let $1\leq j\leq N$ be chosen such that $Y_{j}$ isn't a point. Without
    loss of generality, $Y_{j}\neq X$ because otherwise $X$ would be
    connected and the lemma would be trivial. Since $M_{t}$ is definably
    connected, $bd(Y_{j})\neq\emptyset$. In addition, $Y_{j}$ is clopen
    in $X$ so $bd(Y_{j})\subset bd(X)$. Therefore, there exists some
    $1\leq i\leq n$ such that $a_{i}\in\bar{Y_{i}}$.

    We'll now see that $bd(U_{i})\cap Y_{j}\neq\emptyset$. 

    Assume for contradiction that $bd(U_{i})\cap Y_{j}=\emptyset$. Then
    both $U_{i}\cap Y_{j}$ and $U_{i}^{c}\cap Y_{j}$ are non-empty clopen
    subsets of $X$, which is a contradiction to the fact that $Y_{j}$
    is a connected component. 

    This completes the claim.
  \end{proof}
  Without loss of generality, let's choose an integer $L$ between $1$
  and $N$ such that $\{Y_{1},\dots,Y_{L}\}$ are points and $\{Y_{L+1},\dots,Y_{N}\}$
  are not points. Furthermore, let's choose an integer $M$ between
  $1$ and $n$ such that $\{a_{1},\dots,a_{M}\}$ are isolated and
  $\{a_{M+1},\dots,a_{n}\}$ are not. It's clear that $L\leq M$.

  According to the claim, for each index $L+1\leq j\leq N$ there exists
  an integer $1\leq i\leq n$ and a point $y_{j}$ such that $y_{j}\in bd(U_{i})$
  and $a_{i}\in\bar{Y_{j}}$. We note that from the fact that $a_{i}\in\bar{Y_{j}}$,
  it follows that $a_{i}$ is not an isolated point. This gives us a
  mapping: \[
  \phi:\{Y_{L+1},\dots,Y_{N}\}\rightarrow\bigcup_{i=M+1}^{n}bd(U_{i})\]
  Since $y_{k}\neq y_{l}$ for each $L+1\leq k<l\leq N$ , the map $\phi$
  is injective. Furthermore, \[
  \vert\bigcup_{i=M+1}^{n}bd(U_{i})\vert\leq(n-M)\cdot K\leq(n-L)\cdot K\]
  so by the injectivity of $\phi$ we get that $N-L\leq(n-L)\cdot K$.
  But $K\geq1$ so $N\leq n\cdot K$.\end{proof}
\begin{prop}
  Let $(M,\phi)$ be a $1$-dimensional $\omega$-saturated t.t.t structure
  and let $\alpha(x,y_{1},\dots,y_{l})\in L$ be a formula. Then there exists a constant
  $C\in\mathbb{N}$ such that for every $l$-tuple $c_{1},\dots,c_{l}\in M$,
  $d_{M}(\alpha^{M_{t}}(c_{1},\dots,c_{l}))<C$. \label{prop:bounded-connected}\end{prop}
\begin{proof}
  First of all, let $K\in\mathbb{N}$ a number such that for each point $b\in M_{b}$
  we have $\vert bd(\phi^{M_{t}}(b)\vert\leq K$. By
  lemma \ref{lem:bounded-boundary}, for each $\bar{c}\in M^{l}$ there
  exists a number $n_{\bar{c}}\in\mathbb{N}$ such that $bd(\alpha^{M_{t}}(\bar{c}))<n_{c}$.
  Therefore, since $M$ is $\omega$-saturated, there exists some $n\in\mathbb{N}$
  such that for each tuple $\bar{c}\in M^{l}$, $bd(\alpha^{M_{t}}(\bar{c}))<n$.

  We'll show that we can choose $C$ to be $d_{M}(M_{t})\cdot K\cdot n$.
  Let $m=d_{M}(M_{t})$ and let $\{Y_{1},\dots,Y_{m}\}$ be pairwise
  disjoint definably connected subsets such that $M_{t}=\cup_{i=1}^{m}Y_{i}$.
  By lemma \ref{lem:bounded-connected}, for each $1\leq i\leq m$ and
  $\bar{c}\in M^{l}$, $d_{M}(\alpha^{M_{t}}(\bar{c})\cap Y_{i})<n\cdot K$.
  The proposition then follows immediately.
\end{proof}
We'll now use this boundedness result in order to prove that a certain
set of first order properties are necessary and sufficient for an
$\omega$-saturated first order topological structure to be t.t.t.
\begin{thm}
  Let $(M,\phi)$ be an $\omega$-saturated 1-dimensional t.t.t structure.
  Then $M$ has the following properties:\label{thm:1D-t.t.t-props}
  \begin{enumerate}
  \item For every formula $\alpha(x,y_{1},\dots,y_{l})\in L$, there exists
    some $C\in\mathbb{N}$ such that for every tuple $\overline{c}\in M^{l}$,
    there exist $C$ points $x_{1},\dots,x_{C}$ in $\alpha^{M_{t}}(\overline{c})$
    such that $\alpha^{M_{t}}(\overline{c})\backslash\{x_{1},\dots,x_{C}\}$
    is open.
  \item For every formula $\alpha(x,y_{1},\dots,y_{l})\in L$, there exists
    a constant $C\in\mathbb{N}$ such that for all $c_{1},\dots,c_{l}\in M$,
    $d_{M}(\alpha^{M_{t}}(c_{1},\dots,c_{l}))<C$.
  \item For any pair of formulas $\alpha(x,y_{1},\dots,y_{s})$ and $\beta(x,y_{1},\dots,y_{t})$
    in $L$, and for all $\overline{a}\in M^{s}$ and $\overline{b}\in M^{t}$,
    if $B=\beta^{M_{t}}(\overline{b})\subset\alpha^{M_{t}}(\overline{a})=A$
    is closed and non empty and doesn't have an interior in $A$, then
    $A$ has an interior in $M_{t}$.
  \end{enumerate}
  Furthermore, if $(M,\phi)$ is a first order topological structure which
satisfies these three properties and is Hausdorff, then $M$ is a 
1-dimensional t.t.t structure.
\end{thm}
\begin{proof}
  First we'll see that the three properties are sufficient. Assume that
  $(M,\phi)$ is a first order topological structure
  such that $M_{t}$ is Hausdorff and has the three properties in the
  theorem. 

  By property 1, every definable set $X$ is a boolean combination of
  open sets so $M$ has property (A). By property 2, every definable
  set has a finite number of definably connected components. Finally,
  by property 3, $D(M)=1$.

  Now we'll prove the first part of the theorem. Let $(M,\phi)$ be
  an $\omega$-saturated t.t.t structure. By the definition of t.t.t,
  $M_{t}$ is Hausdorff. We'll now prove that $M$ has each one of the
  required properties.
  \begin{enumerate}
  \item Let $\alpha(x,y_{1},\dots,y_{l})\in L$. Since $M$ is t.t.t, for
    every $\overline{c}\in M^{l}$, there exist $C$ points $x_{1},\dots,x_{C}$
    in $\alpha^{M_{t}}(\overline{c})$ such that $\alpha^{M_{t}}(\overline{c})\backslash\{x_{1},\dots,x_{C}\}$
    is open. Since $M$ is $\omega$-saturated, we can choose $C$ uniformly
    for all $\overline{c}\in M^{l}$. 
  \item This property is essentially proposition \ref{prop:bounded-connected}.
  \item This follows from the fact that $D(M)=1$.
  \end{enumerate}
\end{proof}
\begin{cor}
  Let $\phi(x,y_{1},\dots,y_{k})$ be a formula and let $(M,\phi)$
  be a $1$-dimensional t.t.t structure which is $\omega$-saturated.
  In addition, let $N$ be a model such that $N\equiv M$. Then $(N,\phi)$
  is a $1$-dimensional t.t.t structure.\end{cor}
\begin{proof}
  This is immediate from the fact that all of the properties in theorem
  \ref{thm:1D-t.t.t-props} can be expressed in first order logic.
\end{proof}

\section{Structures With Splitting \label{sec:Structures-With-Splitting}}

\subsection{Introduction}

Our main result in this section is that for any 1-dimensional $\omega$-saturated
t.t.t structure, if removing any point splits the space into more
that one connected component then there exists a finite set $X\subset M_{t}$
such that each connected component of $M_{t}\backslash X$ is o-minimal. 

In order to prove this, we first obtain some intermediate results
such as the fact that the equivalence relation $y\sim_{x}z$ specifying
if $y$ and $z$ are in the same connected component of $M_{t}\backslash\{x\}$
is a definable relation in $M_{t}^{3}$. We also introduce a notion
of {}``local flatness'' which is used as a stepping stone between
t.t.t structures and o-minimality.

For example, consider the structure $R_{int}=\langle\mathbb{R},I(x,y,z)\rangle$
where $I(x,y,z)$ is true if $z$ lies on the interval between $x$
and $y$. In example \ref{exa:interval-example} we'll show that $R_{int}$
has the property that removing any point splits the space into more
that one connected component. We'll use this fact to show that the
relation $y\sim_{x}z$ is indeed definable. In the end of this section
we'll demonstrate how applying the construction of the order to $R_{int}$
gives the standard ordering on the reals. 

In this section we're assuming that $M_{t}$ has no isolated points.
This doesn't pose a problem because $M_{t}$ has at most a finite
number of isolated points so we can remove them without affecting
any of our results.

\subsection{A Definable Relation}

The following equivalence relation is useful for analyzing what happens
when a point is removed from a structure.
\begin{defn}
  Let $M$ be a 1-dimensional t.t.t structure. Let $x,a,b\in M_{t}$.
  Then $a\sim_{x}b$ will be a relation which is true iff $a$ and $b$
  are in the same definable connected component of $M_{t}\backslash\{x\}$.\end{defn}
\begin{rem*}
  Note that by proposition \ref{prop:bounded-connected}, there exists
  an $N\in\mathbb{N}$ such that for each point $x\in M_{t}$, $\sim_{x}$
  has less than $N$ equivalence classes. 

  Our first goal is to show that if for every $x\in M_{t}$ we have
  $d_{M}(M_{t}\backslash\{x\})\geq2$, then $\sim_{x}\subset M_{t}^{3}$
  is definable.
\end{rem*}
We start by showing that for any $x$ such that $d_{M}(M_{t}\backslash\{x\})>2$,
$x\in acl(\emptyset)$.

The following technical lemma will be used many times throughout the
proof.

Intuitively, the lemma says that after removing two points, the space
is divided into three distinct components. The part {}``in between''
the points we removed and one additional side for each of the points.

\begin{lem}
  Let $M$ be a 1-dimensional $\omega$-saturated t.t.t structure, $C\subset M_{t}$
  an open connected definable subset, $a\neq b\in C$ and $2\leq k,l\in\mathbb{N}$
  such that $d_{M}(C\backslash\{a\})=k$ and $d_{M}(C\backslash\{b\})=l$.
  Let $A_{1},\dots,A_{k}$ and $B_{1},\dots,B_{l}$ be the connected
  components of $C\backslash\{a\}$ and $C\backslash\{b\}$ respectively
  such that $a\in B_{1}$ and $b\in A_{1}$. Then:\label{lem:disjoint-components}
  \begin{enumerate}
  \item $bd(\cup_{i=2}^{k}A_{i})=\{a\}$
  \item $bd(\cup_{j=2}^{l}B_{j})=\{b\}$
  \item $bd(A_{1}\cap B_{1})=\{a,b\}$ and for every open set $U$ containing
    $a$ or $b$, $U\cap(A_{1}\cap B_{1})\neq\emptyset$.
  \item The following union is disjoint:\[
    C=(\cup_{i=2}^{k}A_{i})\cup\{a\}\cup(A_{1}\cap B_{1})\cup\{b\}\cup(\cup_{j=2}^{l}B_{j})\]

  \end{enumerate}
\end{lem}
\begin{proof}
  First we'll prove $1$.

  Since $C$ is connected, $a\in\overline{A_{i}}$ for each $1\leq i\leq k$.
  Therefore, since $a\notin A_{i}$ for each $2\leq i\leq k$, $a\in bd(\cup_{i=2}^{k}A_{i})$.
  In addition, $\cup_{i=2}^{k}A_{i}$ is open which means that $bd(\cup_{i=2}^{k}A_{i})\subset\{a\}\cup A_{1}$.
  But $A_{1}$ is open as well and disjoint to $\cup_{i=2}^{k}A_{i}$
  so $bd(\cup_{i=2}^{k}A_{i})=\{a\}$.

  Similarly, $bd(\cup_{j=2}^{l}B_{j})=\{b\}$.

  We'll now show that $\cup_{i=2}^{k}A_{i}\subset B_{1}$. In order
  to do that we first prove that $(\cup_{i=2}^{k}A_{i})\cup\{a\}$ is
  connected in $C\backslash\{b\}$. Assume for contradiction that $X_{1}$
  and $X_{2}$ form a clopen partition of $(\cup_{i=2}^{k}A_{i})\cup\{a\}$
  in $C\backslash\{b\}$. Without loss of generality, $a\in X_{1}$
  which means that $a\notin\overline{X_{2}}$. Furthermore, $A_{1}$
  is open and $b\in A_{1}$ which means that $b\notin\overline{X_{2}}$.
  Together this means that $X_{2}\subset(\cup_{i=2}^{k}A_{i})$ and
  $b\notin bd(X_{2})$. Therefore, $X_{2}$ is clopen in $C$ which 
  is a contradiction to the fact that $C$ is connected.

  Now, $B_{1}$ is the connected component of $C\backslash\{b\}$ containing
  $a$. So from the fact that $(\cup_{i=2}^{k}A_{i})\cup\{a\}$ is connected
  in $C\backslash\{b\}$ it follows that $\cup_{i=2}^{k}A_{i}\subset B_{1}$. 

  We're now ready to prove $4$. It's immediate that \[
  (\cup_{i=2}^{k}A_{i})\cap(A_{1}\cap B_{1})=(\cup_{j=2}^{l}B_{j})\cap(A_{1}\cap B_{1})=\emptyset\]
  In addition, since $(\cup_{i=2}^{k}A_{i})\subset B_{1}$ it follows
  that $(\cup_{j=2}^{l}B_{j})\cap(\cup_{i=2}^{k}A_{i})=\emptyset$.
  This shows that the union is disjoint so all that's left is to show
  that it's equal to $C$.

  Let $c\in C$ be a point such that \[
  c\notin(\cup_{j=2}^{l}B_{j})\cup\{a\}\cup(\cup_{i=2}^{k}A_{i})\cup\{b\}\]
  Since $c\in C\backslash\{a\}$ and $c\notin(\cup_{i=2}^{k}A_{i})$
  it follows that $c\in A_{1}$. Similarly, $c\in B_{1}$. Therefore,
  $c\in A_{1}\cap B_{1}$.

  We're now ready to prove $3$.

  First of all, assume for contradiction that $A_{1}\cap B_{1}=\emptyset$.
  Then by parts 1, 2 and 4 of the lemma, the sets $(\cup_{i=2}^{k}A_{i})\cup\{a\}$
  and $\{b\}\cup(\cup_{j=2}^{l}B_{j})$ would form a clopen partition
  of $C$ which is a contradiction to the assumption that $C$ is connected.

  We'll now show that $bd(A_{1}\cap B_{1})=\{a,b\}$ On the one hand,
  $a\in int(B_{1})$ and $a\in\overline{A_{1}}$ so $a\in bd(A_{1}\cap B_{1})$.
  Similarly, $b\in A_{1}\cap B_{1}$. On the other hand, $A_{1}\cap B_{1}$
  is open so \[
  bd(A_{1}\cap B_{1})\subset(\cup_{i=2}^{k}A_{i})\cup\{a\}\cup\{b\}\cup(\cup_{j=2}^{l}B_{j})\]
  But $\cup_{i=2}^{k}A_{i}$ and $\cup_{j=2}^{l}B_{j}$ are open well
  so $bd(A_{1}\cap B_{1})\subset\{a\}\cup\{b\}$. Together we get that
  $bd(A_{1}\cap B_{1})=\{a,b\}$.

  Finally, let $U$ be an open set containing $a$ or $b$. Since $bd(A_{1}\cap B_{1})=\{a,b\}$,
  it follows that $U\cap(A_{1}\cap B_{1})\neq\emptyset$.
\end{proof}

\begin{lem}
  Let $M$ be a 1-dimensional $\omega$-saturated connected t.t.t structure.
  Let $D\subset M_{t}$ be an open definable subset, $E(x,a,b)\subset M_{t}^{3}$
  a definable relation and $N\in\mathbb{N}$ such that:\label{lem:infinite-sep}
  \begin{enumerate}
  \item $N\geq2$.
  \item For every $x\in D$ and $a,b\in M_{t}$, $a\sim_{x}b\Rightarrow E(x,a,b)$.
  \item For every $x\in D$, $E(x,a,b)$ is an equivalence relation with $N$
    classes.
  \end{enumerate}
  Then for each point $a\in D$, there exists a point $b\in D$ such
  that the definable set $X=\{x\in D\vert\neg E(x,a,b)\}$ is infinite.\end{lem}
\begin{proof}
  Let $a\in D$. Without loss of generality, $D$ is connected. Otherwise,
  we'll look at the connected component containing $a$. We'll now show
  that there exists a point $b\in D$ such that for an infinite number
  of points $x\in D$ we have $\neg E(x,a,b)$. In order to do this,
  we'll inductively construct a sequence of points $(b_{1},b_{2},\dots)$
  in $D$ such that for each $n\in\mathbb{N}$ and each $1\leq j<n$,
  $a\sim_{b_{n}}b_{j}$ and $\neg E(b_{j},a,b_{n})$.

  For $n=1$, we can choose any $b_{1}\in D\backslash\{a\}$.

  Let's assume that we constructed the sequence up to $b_{n}$. Let
  $X_{1},\dots,X_{c(b_{n})}$ be the connected components of $M_{t}\backslash\{b_{n}\}$
  such that $a\in X_{2}$. We choose $b_{n+1}\in D$ to be some point
  such that $\neg E(b_{n},a,b_{n+1})$. By our assumptions on $E(x,a,b)$,
  $b_{n+1}\notin X_{2}$. So without loss of generality, in $b_{n+1}\in X_{1}$.
  Let $Y_{1},\dots,Y_{c(b_{n+1})}$ be the connected components of $M\backslash\{b_{n+1}\}$
  such that $b_{n}\in Y_{1}$. By lemma \ref{lem:disjoint-components},
  for all $1<j\leq c(b_{n+1})$, $Y_{j}\cap X_{2}=\emptyset$. By the
  inductive hypothesis, $b_{j}\in X_{2}$ for all $1\leq j<n$. This
  means that for all $1\leq j<n$, $b_{j}\in Y_{1}$. Similarly, $a\in Y_{1}$
  and we already know that $b_{n}\in Y_{1}$. Together we've shown that
  $a\sim_{b_{n+1}}b_{j}$ for all $1\leq j<n+1$.

  We'll now show that for all $1\leq j<n$, $\neg E(b_{j},a,b_{n+1})$.
  This will be enough because we already know that $\neg E(b_{n},a,b_{n+1})$.

  Let $j$ be an index such that $1\leq j<n$. Let $X_{1},\dots,X_{c(b_{j})}$
  be the connected components of $M_{t}\backslash\{b_{j}\}$ such that
  $a\in X_{2}$ and $b_{n}\in X_{1}$. In addition, let $Y_{1},\dots,Y_{c(b_{n})}$
  be the connected components of $M_{t}\backslash\{b_{n}\}$ such that
  $b_{j},a\in Y_{1}$ and $b_{n+1}\in Y_{2}$. Since $b_{j}\in Y_{1}$
  and $b_{n}\in X_{1}$, by lemma \ref{lem:disjoint-components} it
  follows that $Y_{2}\subset X_{1}$ which means that $b_{n}\sim_{b_{j}}b_{n+1}$.
  By our assumptions on $E$ this implies $E(b_{j},b_{n},b_{n+1})$.
  Therefore, as $\neg E(b_{j},a,b_{n})$, we can conclude that $\neg E(b_{j},a,b_{n+1})$. 

  Now, by the $\omega$-saturation, there exists a point $b\in D$ such that
  $\vert\{x\in D:\neg E(x,a,b)\}\vert=\infty$. \end{proof}
\begin{lem}
  Let $M$ be a 1-dimensional $\omega$-saturated connected t.t.t structure.
  Let $D\subset M_{t}$ be an open definable subset, $E(x,a,b)\subset M_{t}^{3}$
  a definable relation and $N\in\mathbb{N}$ such that:
  \begin{enumerate}
  \item For every $x\in D$ and $a,b\in M_{t}$, $a\sim_{x}b\Rightarrow E(x,a,b)$.
  \item For every $x\in D$, $E(x,a,b)$ is an equivalence relation with $N$
    classes.
  \end{enumerate}
  Then $N\leq2$. \label{lem:not-all-connected-k}\end{lem}
\begin{proof}
  Assume for contradiction that $N>2$. For ease of notion, we define
  $c(x)=d_{M}(M_{t}\backslash\{x\})$ for each $x\in M_{t}$. We note
  that for all $x\in D$, $c(x)>2$.

  By lemma \ref{lem:infinite-sep}, there exist points $a,b\in M_{t}$
  such that for an infinite number of points $x\in D$, $\neg E(x,a,b)$.

  We denote the infinite set $\{x\in D:\neg E(x,a,b)\}$ by $X$.

  Let $x,y\in X$, let $X_{1},\dots,X_{c(x)}$ be the connected components
  of $M_{t}\backslash\{x\}$ such that $a\in X_{1}$ and $b\in X_{2}$
  and let $Y_{1},\dots,Y_{c(y)}$ be the connected components of $M_{t}\backslash\{y\}$
  such that $a\in Y_{1}$ and $b\in Y_{2}$. 

  First we note that for every $j$ such that $3\leq j\leq c(x)$,
  $y\notin X_{j}$. Because let $j$ be an index such that $3\leq j\leq n(x)$,
  let $k$ be an index such that $x\in Y_{k}$ and assume for contradiction
  that $y\in X_{j}$. Then, since $y\in X_{j}$ and $x\in Y_{k}$, it
  follows from lemma \ref{lem:disjoint-components} that $X_{1},X_{2}\subset Y_{k}$
  which means that $a,b\in Y_{k}$. However, this contradicts our assumption
  that $a\nsim_{y}b$.

  In an analogous fashion, for each index $j$ such that $3\leq j\leq c(y)$
  we have $x\notin Y_{j}$. 

  Therefore, $x\in Y_{1}\cup Y_{2}$ and $y\in X_{1}\cup X_{2}$.

  Since this is true for any pair of points $x,y\in X$, by lemma \ref{lem:disjoint-components}
  we get that:
  \begin{enumerate}
  \item for all $3\leq i\leq c(x)$ and $3\leq j\leq c(y)$ , $X_{i}\cap Y_{j}=\emptyset$
  \item for all $3\leq i\leq c(x)$, $X_{i}\cap X=\emptyset$.
  \end{enumerate}
  From these two results we'll show that that $M_{t}\backslash X$ has
  an infinite number of definable connected components. First of all,
  by $2$ it follows that for each point $x\in X$, the classes of $E(x,a,b)$
  not containing $a$ and $b$ are definable sets which are contained
  and clopen in $M_{t}\backslash X$. Furthermore, by $1$, all the
  sets obtained this way are disjoint. 

  But since $X$ is both infinite and definable, this is a contradiction
  to the fact that $M$ is t.t.t.
\end{proof}
Now, let $p$ be some type in $S(\emptyset)$. We now show that there
exist a $\emptyset$-definable relation $R_{p}(x,a,b)\subset M_{t}^{3}$
and an infinite $\emptyset$-definable set $D_{p}\subset M_{t}$ such
that:
\begin{enumerate}
\item For all elements $x\models p$ and points $a,b\in M_{t}$, $R_{p}(x,a,b)\iff a\sim_{x}b$.
\item For all elements $x\in D_{p}$ and points $a,b\in M_{t}$, $a\sim_{x}b\Rightarrow R_{p}(x,a,b)$.
\item For all elements $x\in D_{p}$, $R_{p}(x,a,b)\subset M_{t}^{2}$ is
  an equivalence relation with $d_{M}(M_{t}\backslash\{y\})$ equivalence
  classes where $y$ is some element realizing $p$.
\item For every element $x$ that realizes $p$, $x\in D_{p}$. 
\end{enumerate}
We construct $R_{p}$ and $D_{p}$ in the following way. First, let
$x$ be some realization of $p$ and define $N$ by $N=d_{M}(M_{t}\backslash\{x\})$.
Then there exist $\phi_{1}(x,\bar{y}),\dots,\phi_{N}(x,\bar{y})$
such that:

\medskip{}

({*}) for some $\bar{y}$, $\phi_{1}^{M_{t}}[\bar{y}],\dots,\phi_{N}^{M_{t}}[\bar{y}]$
partition $M_{t}\backslash\{x\}$ into $N$ disjoint clopen sets.
Furthermore, for any other $\bar{z}$, if $(\phi_{1}^{M_{t}}[\bar{z}],\allowbreak \dots,\allowbreak \phi_{N}^{M_{t}}[\bar{z}])$
is a partition of $M_{t}\backslash\{x\}$ into disjoint clopen sets
then it's the same partition as $(\phi_{1}^{M_{t}}[\bar{y}],\allowbreak \dots,\allowbreak \phi_{N}^{M_{t}}[\bar{y}])$.

\medskip{}

Since this is a first order statement, ({*}) holds for all $x\models p$.

Now, we define $D_{p}$ as the set of all the points $x\in M_{t}$
such that ({*}) holds for $x$ with the formulas $\phi_{1}(x,\bar{y}),\dots,\phi_{N}(x,\bar{y})$.
We then define $R_{p}(x,a,b)$ as a relation which is true iff for
one of the points $\bar{y}$ guaranteed by ({*}) for $x$, the sets
$\phi_{1}^{M_{t}}[\bar{y}],\dots,\phi_{N}^{M_{t}}[\bar{y}]$ partition
$M_{t}\backslash\{x\}$ into $N$ disjoint clopen sets such that $a$
and $b$ are in the same section of the partition.
\begin{prop}
  Let $M$ be a 1-dimensional $\omega$-saturated connected t.t.t structure
  and let $x\in M_{t}$ be a point such that $d_{M}(M_{t}\backslash\{x\})>2$.
  Then $D_{tp(x/\emptyset)}$ is finite and in particular, $x\in acl(\emptyset)$.\label{prop:bigger-2-acl}\end{prop}
\begin{proof}
  Let $N=d_{M}(M_{t}\backslash\{x\})$ and $p=tp(x/\emptyset)$. Since
  $N>2$, by applying lemma \ref{lem:not-all-connected-k} with $D=int(D_{p})$
  and $E=R_{p}$, $int(D_{p})$ is finite. Therefore, $D_{p}$ is finite.
\end{proof}
We now look at what happens if $d_{M}(M_{t}\backslash\{x\})=2$.

As before, let $p\in S(\emptyset)$ be a type such that for some element
$x$ realizing $p$ we have $d_{M}(M_{t}\backslash\{x\})=2$. We define
$\tilde{D}_{p}\subset D_{p}$ as the set of points $x\in D_{p}$ such
that there exist elements $a,b\in M_{t}$ and a basis set $U\subset M_{t}$
containing $x$ such that for all points $u\in U$, $\neg R_{p}(u,a,b)$.
\begin{prop}
  Let $M$ be a 1-dimensional $\omega$-saturated connected t.t.t structure
  and let $p\in S_{1}(\emptyset)$ be a complete type in $M_{t}$. In
  addition, assume that for some (all) elements $x\models p$, $\sim_{x}$
  has $2$ equivalence classes. Then, for each point $x$ realizing
  $p$, one of the following hold:\label{prop:equals-2}
  \begin{enumerate}
  \item There exists a finite $\emptyset$-definable subset of $D_{p}$ containing
    $x$ and in particular, $x\in acl(\emptyset)$.
  \item $int(\tilde{D}_{p})$ is a set containing $x$ such that for every
    point $y\in int(\tilde{D}_{p})$, $d_{M}(M_{t}\backslash\{y\})=2$.
  \end{enumerate}
\end{prop}
\begin{proof}
  First of all, if $D_{p}$ is finite then the first case holds for
  all $x\models p$.

  Let's assume that $D_{p}$ is infinite. Now, suppose that for each
  point $x$ realizing $p$:

  \medskip{}

  ({*}) for all $a,b\in M_{t}$ and for every basis set $U$ containing
  $x$, there exists a point $u\in U$ such that $R(u,a,b)$.

  \medskip{}

  We define the set $C\subset D_{p}$ as the set of points in $D_{p}$
  with property ({*}). $C$ is clearly $\emptyset$-definable. Furthermore,
  for all $x\models p$, $x\in C$. Assume for contradiction that $C$
  is infinite. In that case, by lemma \ref{lem:infinite-sep} there
  exist points $a,b\in M_{t}$ and a basis set $U\subset C$ such that
  for all $u\in U$, $\neg R(u,a,b)$. This is clearly a contradiction
  to ({*}). This means that $C$ is finite so again we're in the first
  case for every point $x\models p$.

  Therefore, we can assume that for all elements $x$ realizing $p$,
  $x\in\tilde{D}_{p}$. 

  If $\tilde{D}_{p}$ is finite then again we're in the first case for
  every point $x\models p$. 

  We'll now see that if $\tilde{D}_{p}$ is infinite then for each point
  $y\in int(\tilde{D}_{p})$, $d_{M}(M_{t}\backslash\{y\})\leq2$. This
  will finish the proposition because we already know that for every
  point $x\in\tilde{D}_{p}$, $d_{M}(M_{t}\backslash\{x\})\geq2$. We
  also note that if $x\in\tilde{D_{p}}\backslash int(\tilde{D}_{p})$
  then clearly we're in the first case as $\vert\tilde{D}_{p}\backslash int(\tilde{D}_{p})\vert<\infty$. 

  Let's assume for contradiction that $y\in int(\tilde{D}_{p})$ and
  $d_{M}(M_{t}\backslash\{y\})>2$. Since $y\in\tilde{D}_{p}$, there
  exist points $a,b\in M_{t}$ and a basis set $U\subset\tilde{D}_{p}$
  containing $y$ such that for all $u\in U$, $\neg R_{p}(u,a,b)$.
  Let $Y_{1},Y_{2},Y_{3}$ be three definable disjoint clopen sets partitioning
  $M_{t}\backslash\{y\}$ such that $a\in Y_{1}$ and $b\in Y_{2}$.
  Since $M_{t}$ is connected, there exists a $z$ such that $z\in Y_{3}\cap U$.
  Since $z\in D_{p}$, $k=d_{M}(M_{t}\backslash\{z\})\geq2$. Let $Z_{1},\dots,Z_{k}$
  be definable pairwise disjoint clopen sets partitioning $M_{t}\backslash\{z\}$
  such that $y\in Z_{1}$. Since $y\in Z_{1}$ and $z\in Y_{3}$, by
  lemma \ref{lem:disjoint-components} it follows that $Y_{1}\cup Y_{2}\subset Z_{1}$
  which means that $a,b\in Z_{1}$. However, this is a contradiction
  to the fact that $\neg R_{p}(z,a,b)$. 
\end{proof}
We now use the previous two propositions to show that if for each
point $x\in M_{t}$ we have $d_{M}(M_{t}\backslash\{x\})>1$, then
the relation $a\sim_{x}b\subset M_{t}^{3}$ is $\emptyset$-definable.
\begin{prop}
  Let $M$ be a 1-dimensional connected $\omega$-saturated t.t.t structure such
  that for each point $x\in M_{t}$, $d_{M}(M_{t}\backslash\{x\})>1$.
  Then the relation $a\sim_{x}b\subset M_{t}^{3}$ is definable.\label{prop:definable-relation}\end{prop}
\begin{proof}
  We'll show that both $a\sim_{x}b$ and $a\nsim_{x}b$ are $\bigvee$-definable
  by formulas without parameters.

  $a\nsim_{x}b$ is clearly $\bigvee$-definable by formulas without
  parameters because $a\nsim_{x}b$ iff there exist two open sets whose
  boundary is $\{x\}$ such that one contains $a$ and the other contains
  $b$.

  We'll now prove that $a\sim_{x}b$ is $\bigvee$-definable by formulas
  without parameters. This is done by showing that for each point $x\in M_{t}$,
  there exists a set $C_{x}\subset M_{t}^{3}$ which is definable without
  parameters such that:
  \begin{enumerate}
  \item For all points $y,a,b\in M_{t}$, $(y,a,b)\in C_{x}\Rightarrow a\sim_{y}b$
  \item $(x,a,b)\in C_{x}\iff a\sim_{x}b$
  \end{enumerate}
  Let's choose some $x\in M_{t}$ and define $p=tp(x/\emptyset)$.

  If $d_{M}(M_{t}\backslash\{x\})=N>2$ then, by proposition \ref{prop:bigger-2-acl},
  $D_{p}$ is a finite $\emptyset$-definable set containing $x$. Furthermore,
  for every $y\in D_{p}$, $d_{M}(M_{t}\backslash\{x\})\geq N$. Let's
  denote the points in $D_{p}$ by $D_{p}=\{y_{1},\dots,y_{k}\}$. Without
  loss of generality, there exists some $0\leq l<k$ such that for all
  $1\leq i\leq l$, $d_{M}(M_{t}\backslash\{y_{i}\})>N$ and for all
  $l+1\leq i\leq k$, $d_{M}(M_{t}\backslash\{y_{i}\})=N$. It's easy
  to see that for each $1\leq i\leq l$, $x\notin D_{tp(y_{i}/\emptyset)}$.
  Therefore, we can define:\[
  C_{x}=((D_{p}\backslash\bigcup_{i=1}^{l}D_{tp(y_{i}/\emptyset)})\times M_{t}^{2})\cap R_{p}\]
  . Finally, let's assume that $d_{M}(M_{t}\backslash\{x\})=2$. 

  If $D_{p}$ contains a finite $\emptyset$-definable set containing
  $x$, then we can define $C_{x}$ in the same way as in the previous
  case. Otherwise, by proposition \ref{prop:equals-2}, $int(\tilde{D}_{p})\subset D_{p}$
  is a set containing $x$ such that for all $y\in int(\tilde{D}_{p})$,
  $d_{M}(M_{t}\backslash\{y\})=2$. Therefore, we can define:

  \[
  C_{x}=(int(\tilde{D}_{p})\times M_{t}^{2})\cap R_{p}\]
  . This finishes the proof of the proposition.\end{proof}
\begin{example}
  Let's look at the structure $R_{int}=\langle\mathbb{R},I(x,y,z)\rangle$
  where $I(x,y,z)$ is true if $z$ lies on the interval between $x$
  and $y$. In other words:\label{exa:interval-example}\[
  I(x,y,z)=\{(x,y,z)\in\mathbb{R}^{3}\vert(x<z<y)\vee(y<z<x)\}\]

  The basis sets will be given by:

  \[
  \{I^{R_{int}}(a,b,z)\vert a,b\in\mathbb{R}\}\]

  Since $\langle\mathbb{R},<\rangle$ is an $\omega$-saturated one dimensional
  t.t.t structure, so is $R_{int}$. We'll now see that for every point
  $a\in\mathbb{R}$, $\mathbb{R}\backslash\{a\}$ has two definably
  connected components in $R_{int}$.

  Let $a$ be some point in $\mathbb{R}$. Let $c$ and $b$ be two
  constants in $\mathbb{R}$ such that $c<a<b$. Then:\[
  x<a\iff I(x,a,c)\vee I(c,a,x)\]

  \[
  a<x\iff I(a,x,b)\vee I(a,b,x)\]
  this shows that $\mathbb{R}\backslash\{a\}$ has two definably connected
  components in $R_{int}$.

  Therefore, by proposition \ref{prop:definable-relation}, the relation
  $a\sim_{x}b$ is definable in $R_{int}$. Indeed:\[
  a\sim_{x}b\iff\neg I(a,b,x)\]

\end{example}

\subsection{Local and Global Flatness}

We'll now prove that, under the condition that removing any point
creates at least two connected components, there exist a finite number
of points such that after removing them, the remaining finite number
of connected components are o-minimal. This is done by first showing
that up to a finite number of points the structure is ''locally
o-minimal'', and then showing that local o-minimality implies global
o-minimality. In addition, the definability of the relation $a\sim_{x}b\subset M_{t}^{3}$
will play a crucial role. 

We start by defining a notion of {}``local flatness'' and then showing
that locally flat points have a neighborhood which behaves similarly
to an o-minimal one.
\begin{defn}
  Let $M$ be t.t.t structure. We say that the point $x\in M_{t}$ is \emph{locally
    flat} if there exist points $a,b\in M_{t}$ and a basis set $U$ such
  that for every point $u\in U$, $a\nsim_{u}b$. We say that a set $D$ is
  locally flat if all of it's points are locally flat.

  We first show that in the type of structures which we're currently
  interested in, all but a finite number of points are locally flat.\end{defn}
\begin{prop}
  Let $M$ be a 1-dimensional $\omega$-saturated connected t.t.t structure
  such that for each point $x\in M_{t}$, $d_{M}(M_{t}\backslash\{x\})>1$.
  Then for all but a finite number of points $x\in M_{t}$, $x$ is locally
  flat.\label{prop:finite-not-flat}\end{prop}
\begin{proof}
  Let $X\subset M_{t}$ be the set of points $x\in M_{t}$ such that
  for all points $a,b\in M_{t}$ and for every basis set $U$, there
  exists a point $y\in U$ such that $a\sim_{y}b$. By proposition \ref{prop:definable-relation},
  $X$ is definable.

  Assume for contradiction that $X$ is infinite. Then, by lemma \ref{lem:infinite-sep},
  there exists a pair of points $a,b\in M_{t}$ such that the set $\tilde{X}=\{x\in X\vert a\nsim_{x}b\}$
  is infinite. In addition, $\tilde{X}$ is definable so there exists
  a basis set $U$ which is contained in $\tilde{X}\subset X$. This
  is clearly a contradiction to the definition of $X$. 
\end{proof}
The next few propositions will show that points which are locally
flat have a neighborhood on which we can define a linear order. This
motivates the {}``flatness'' in the definition.
\begin{lem}
  Let $M$ be a 1-dimensional $\omega$-saturated connected t.t.t structure
  such that for all $x\in M_{t}$, $d_{M}(M_{t}\backslash\{x\})>1$.
  Let $x\in M_{t}$ be locally flat and let $U\subset M_{t}$ be an
  open connected definable set containing $x$. Then $U\backslash\{x\}$
  has two connected components.\label{lem:flat-two-components}\end{lem}
\begin{proof}
  Since $d_{M}(M_{t}\backslash\{x\})>1$ and $M_{t}$ is connected,
  it's enough to show that $U\backslash\{x\}$ has no more than two
  connected components. 

  Assume for contradiction that $U_{1},\dots,U_{k}$ are the connected
  components of $U\backslash\{x\}$ with $k>2$. In addition, let $a,b\in M_{t}$
  be points and let $V\subset U$ be a basis set containing $x$ such
  that for all $v\in V$, $a\nsim_{v}b$.

  Let the sets $X_{1},X_{2}\subset M_{t}\backslash U$ be part of
  a clopen partition of $M_{t}\backslash U$ such that $a\in X_{1}$
  and $b\in X_{2}$. By the connectedness of $M_{t}$, there exist points 
  $b_1, b_2 \in bd(U)$ such that $b_1 \in X_1$ and $b_2 \in X_2$. Without 
  loss of generality, $b_1 \in bd(U_{1})$ and $b_2 \in bd(U_{2})$.
  Furthermore, the connectedness of $M_{t}$ implies that $U_{3}\cap V\neq\emptyset$.

  Now, let $y$ be a point in $U_{3}\cap V$. By lemma \ref{lem:disjoint-components},
  $U_{1}$ and $U_{2}$ are both subsets of the connected component
  of $U\backslash\{y\}$ which contains $x$. So since $\overline{U_{1}}\cap X_{1}\neq\emptyset$
  and $\overline{U_{2}}\cap X_{2}\neq\emptyset$, both $X_{1}$ and $X_{2}$ are in the same
  connected component of $M_{t}\backslash\{y\}$. This means that $a\sim_{y}b$
  which is a contradiction to the fact that $y\in V$.
\end{proof}
One consequence of lemma \ref{lem:flat-two-components} which will
be used later is that if we remove all of the finite number of points
which aren't locally flat then for each of the remaining connected
components $C$, and for each point $x\in C$, $C\backslash\{x\}$
will have exactly two connected components. This will be used to show
that $C$ is o-minimal. 

Our next goal is to define an order on some neighborhood of each locally
flat point. Let $x_{0}\in U$ be locally flat and let $U\subset M_{t}$
be a connected definable neighborhood of $x_{0}$. We define a relation $<_{x_{0},U}$
on $U$ in the following way.

By lemma \ref{lem:flat-two-components}, $U\backslash\{x_{0}\}$ has
two connected components which we'll denote by $V_{+}$ and $V_{-}$.
Let $x$ and $y$ be points in $U$. We'll say that $x<_{x_{0},U}y$
if one of the following hold:
\begin{itemize}
\item $x,y\in V_{+}$ and $x_{0}\nsim_{x}y$
\item $x,y\in V_{-}$ and $x_{0}\nsim_{y}x$
\item $y\in V_{+}$ and $x\in V_{-}$.
\item $y=x_{0}$ and $x\in V_{-}$.
\item $x=x_{0}$ and $y\in V_{+}$.
\end{itemize}

Note that by proposition \ref{prop:definable-relation}, $x<_{x_{0},U}y$  is definable.

We now show that if $x_{0}$ is locally flat then there exists a neighborhood
$x_{0}\in U$ such that $<_{x_{0},U}$ defines a dense linear order
on $U$.
\begin{prop}
  Let $M$ be a 1-dimensional connected $\omega$-saturated t.t.t structure
  such that for all $x\in M_{t}$, $d_{M}(M_{t}\backslash\{x\})>1$.
  Let $D\subset M_{t}$ be a connected open definable subset such that
  for every point $x\in D$, $d_{M}(M\backslash\{x\})=2$. Let $x_{0}\in D$
  be a locally flat point. Then there exists a connected open neighborhood
  $U\subset D$ of $x_{0}$ such that $<_{x_{0},U}$ defines a dense
  linear order on $U$.\label{prop:exist-set-with-order}\end{prop}
\begin{proof}
  Since $x_{0}$ is locally flat, there exist points $a,b\in M_{t}$
  and a basis set $V\subset D$ such that $x_{0}\subset V$ and for
  every point $y\in V$ we have $a\nsim_{y}b$.

  We can assume that $V$ is connected because otherwise we can take
  the connected component containing $x_{0}$. In addition, let $V_{+}$
  and $V_{-}$ be the two connected components of $V\backslash\{x_{0}\}$
  (by lemma \ref{lem:flat-two-components} there are exactly two) and
  let $C_{a}$ and $C_{b}$ be the connected components of $M_{t}\backslash\{x_{0}\}$
  such that $a\in C_{a}$ and $b\in C_{b}$. In addition, without loss
  of generality $V_{+}\subset C_{a}$ and $V_{-}\subset C_{b}$. This
  follows from the fact that $C_{a}\cap V$ and $C_{b}\cap V$ partition
  $V$ into two clopen sets so by lemma \ref{lem:flat-two-components},
  one must equal $V_{+}$ and the other must equal $V_{-}$.

  We'll now show that $<_{x_{0},V}$ is a dense linear order on $V$.

  Let $x$, $y$ and $z$ be points in $V_{+}$. In addition, let $X_{a}$
  and $X_{b}$ be the connected components of $M_{t}\backslash\{x\}$
  such that $a\in X_{a}$ and $b\in X_{b}$. $Y_{a}$, $Y_{b}$, $Z_{a}$
  and $Z_{b}$ are defined analogously for $y$ and $z$.

  Since $x,y\in V_{+}\subset C_{a}$, it follows from lemma \ref{lem:disjoint-components}
  that $x_{0}\in X_{b}$ and $x_{0}\in Y_{b}$. Because if we assume
  for contradiction that $x_{0}\in X_{a}$ then by lemma \ref{lem:disjoint-components}
  we get that $C_{b}\cap X_{b}=\emptyset$ which is a contradiction
  to the fact that $b\in C_{b}\cap X_{b}$. The proof that $x_{0}\in Y_{b}$ is
  identical. 
  \begin{enumerate}
  \item $x\nless_{x_{0},V}y\Rightarrow y<_{x_{0},V}x$:

    According to the assumption, $x_{0}\sim_{x}y$ which together with
    the fact that $x_{0}\in X_{b}$ means that $y\in X_{b}$. Now let's
    assume for contradiction that $x\in Y_{b}$. Since $y\in X_{b}$ we
    get from lemma \ref{lem:disjoint-components} that $Y_{a}\cap X_{a}=\emptyset$
    which is a contradiction since $a\in X_{a}\cap Y_{a}$. Therefore,
    $x\in Y_{a}$ which together with $x_{0}\in Y_{b}$ gives $y<x_{o,V}x$.

  \item $x<_{x_{0},V}y\Rightarrow y\nless_{x_{0},V}x$:

    Since $x<_{x_{0},V}y$ and $x_{0}\in X_{b}$, we get that $y\in X_{a}$.
    Now, we claim that $x\in Y_{b}$. For otherwise, if $x\in Y_{a}$
    then we'd get from lemma \ref{lem:disjoint-components} that $Y_{b}\cap X_{b}=\emptyset$
    which is a contradiction.

    Therefore, since $x_{0}\in Y_{b}$ as well, $y\nless_{x_{0},V}x$.

  \item $x<_{x_{0},V}y\wedge y<_{x_{0},V}z\Rightarrow x<_{x_{0},V}z$:

    According to the assumptions, $x_{0},x\in Y_{b}$, $z\in Y_{a}$,
    $x_{0}\in X_{b}$ and $y\in X_{a}$. We have to prove that $z\in X_{a}$
    as well. But again by lemma \ref{lem:disjoint-components}, $y\in X_{a}\wedge x\in Y_{b}\Rightarrow Y_{a}\subset X_{a}$.

  \end{enumerate}
  The proof of these claims is either trivial or identical when $x$,
  $y$ and $z$ are distributed differently among $V_{+}$, $V_{-}$
  and $\{x_{0}\}$.

  This shows that $<_{x_{0},V}$ is indeed a linear order. We'll now
  show that if $x<_{x_{0},V}y$ then there exists a point $z\in V$
  such that $x<_{x_{0},V}z<_{x_{0},V}y$. Again we'll assume that $x,y\in V_{+}$.
  As before, this means that $x\in Y_{b}$ and $y\in X_{a}$ which by
  lemma \ref{lem:disjoint-components} implies that
  $(X_{a}\cap Y_{b})\cap V\neq\emptyset$. Let $z$ be some point
  in $(X_{a}\cap Y_{b})\cap V$. By the definition of $<_{x_{0},V}$
  and the fact that it's linear we get that $x<_{x_{0},V}z<_{x_{0},V}y$.
\end{proof}
Before extending the order defined to connected components, we introduce
the notion of an interval in a t.t.t structure and prove some useful
properties.
\begin{defn}
  Let $M$ be a 1-dimensional t.t.t structure and $x,y\in M_{t}$ such
  that $d_{M}(M_{t}\backslash\{x\})=d_{M}(M_{t}\backslash\{y\})=2$.
  In addition, let $X_{1}$ and $X_{2}$ be a clopen partition of $M_{t}\backslash\{x\}$
  and let $Y_{1}$ and $Y_{2}$ be a clopen partition of $M_{t}\backslash\{y\}$
  such that $x\in Y_{1}$ and $y\in X_{1}$. Then \emph{the interval
    between $x$ and $y$ }will be defined as $I(x,y)=X_{1}\cap Y_{1}$.
  If $x=y$ then $I(x,y)=\emptyset$.\end{defn}
\begin{rem*}
  By lemma \ref{lem:disjoint-components}, if $V\subset M_{t}$ is an 
  open definable connected subset and $x,y\in V$
  then $I(x,y)\cap V\neq\emptyset$ and the following union is disjoint:
\[
V=(X_{2}\cap V)\cup\{x\}\cup(I(x,y)\cap V)\cup\{y\}\cup(Y_{2}\cap V)\]
This motivates us to think of $I(x,y)$ as the set lying {}``in between''
$x$ and $y$.
\end{rem*}
\begin{lem}
  Let $M$ be a 1-dimensional connected $\omega$-saturated t.t.t structure
  such that for every point $x\in M_{t}$, $d_{M}(M_{t}\backslash\{x\})>1$.
  Let $D\subset M_{t}$ be a connected open definable subset which is
  locally flat. Let $x\neq y$ be points in $D$. Then:\label{lem:interval-lemma}
  \begin{enumerate}
  \item $I(x,y)\cap D$ is a non-empty definable open connected set.
  \item $\{x,y\}=bd(I(x,y))$.
  \item If $a,b\in D\backslash\{x,y\}$ such that $a\sim_{x}b$ and $a,b\notin I(x,y)$,
    then $a\sim_{y}b$.
  \item If $a,b\in D\backslash\{x,y\}$ such that $a\nsim_{x}b$ and $a,b\notin I(x,y)$,
    then $a\nsim_{y}b$.
  \end{enumerate}
\end{lem}
\begin{proof}
  Let $X_{1}$ and $X_{2}$ be the connected components of $M_{t}\backslash\{x\}$
  and let $Y_{1}$ and $Y_{2}$ be the connected components of $M_{t}\backslash\{y\}$.
  Note that by proposition \ref{lem:flat-two-components}, both $D\backslash\{x\}$
  and $D\backslash\{y\}$ have exactly two connected components which
  are given by $X_{1}\cap D$, $X_{2}\cap D$ and $Y_{1}\cap D$, $Y_{2}\cap D$
  respectively. Without loss of generality, $x\in Y_{1}$ and $y\in X_{1}$
  so $I(x,y)=X_{1}\cap Y_{1}$, $X_{2}\subset Y_{1}$ and $Y_{2}\subset X_{1}$.
  By lemma \ref{lem:disjoint-components} this means that

  \medskip{}

  ({*}) $D=(X_{2}\cap D)\cup\{x\}\cup(I(x,y)\cap D)\cup\{y\}\cup(Y_{2}\cap D)$.

  \medskip{}

  We'll now prove the four parts of the lemma.
  \begin{enumerate}
  \item First of all, since $X_{1}$ and $Y_{1}$ are definable, $I(x,y)$
    is definable as well.

    By lemma \ref{lem:disjoint-components}, $I(x,y)\cap D\neq\emptyset$.
    $I(x,y)$ is open as the intersection of open sets.

    Now, assume for contradiction that $I(x,y)\cap D$ isn't connected.
    Let $A_{1}$ and $A_{2}$ be a clopen partition of $I(x,y)\cap D$.
    By lemma \ref{lem:disjoint-components} , the boundaries of $A_{1}$
    and $A_{2}$ in $D$ are contained in $\{x,y\}$. Since $D$ is connected,
    for each $i=1,2$ we have either $x\in bd(a_{i})$ or $y\in bd(A_{i})$.
    Assume for contradiction that $x\notin\overline{A_{1}}$ and $y\notin\overline{A_{2}}$.
    By ({*}) this means that the sets \[
    (D\cap X_{2})\cup\{x\}\cup A_{2},\:(D\cap Y_{2})\cup\{y\}\cup A_{1}\]
    form a clopen partition of $D$ which is a contradiction to the assumption
    that $D$ is connected. 

    Therefore, without loss of generality, we can assume that $x\in bd(A_{1})$
    and $x\in bd(A_{2})$. But then the set $U=D\cap Y_{1}$ is open and
    $U\backslash\{x\}$ has more than two connected components which is
    a contradiction to lemma \ref{lem:flat-two-components}.

  \item This follows immediately from lemma \ref{lem:disjoint-components}.
  \item If $a,b\in X_{2}\cap D$ then since $X_{2}\subset Y_{1}$, $a,b\in Y_{1}$
    which means that $a\sim_{y}b$. So we can assume that $a,b\in X_{1}\cap D$.
    By the assumption, $a,b\notin Y_{1}$ which by ({*}) implies that
    $a,b\in Y_{2}\Rightarrow a\sim_{y}b$.
  \item Without loss of generality, $a\in X_{1}$ and $b\in X_{2}$. Therefore,
    $b\in Y_{1}$. In addition, $Y_{2}\subset X_{1}$ and $a\notin Y_{1}\cap X_{1}$
    so $a\in Y_{2}$. This means that $a\nsim_{y}b$.
  \end{enumerate}
\end{proof}
\begin{lem}
  Let $M$ be a 1-dimensional connected $\omega$-saturated t.t.t structure
  such that for all $x\in M_{t}$, $d_{M}(M_{t}\backslash\{x\})>1$.
  Let $D\subset M_{t}$ be a connected open definable subset which is
  locally flat. Let $x_{0}\in D$. Then there exist points $a,b\in D$
  such that:\label{lem:exists-interval-with-order}
  \begin{enumerate}
  \item $x_{0}\in I(a,b)$ 
  \item $<_{x_{0},I(a,b)}$ defines a dense linear order on $I(a,b)$
  \item $I(a,b)\subset D$
  \item For all $x\in I(a,b)$, $a\nsim_{x}b$
  \end{enumerate}
\end{lem}
\begin{proof}
  By proposition \ref{prop:exist-set-with-order}, there exists a definable
  connected open neighborhood $U\subset D$ of $x_{0}$ such that $<_{x_{0},U}$
  defines a dense linear order on $U$. Since $M_{t}$ is Hausdorff
  and $bd(U)$ is finite, we can assume that $<_{x_{0},U}$ defines
  a dense linear order on $\overline{U}$. Let $a$ be the point on
  $bd(U)$ such that $x_{0}<_{x_{0},U}a$ and such that for every point
  $y\in bd(U)$ with $x_{0}<_{x_{0},U}y$, $a\leq_{x_{0},U}y$. Similarly,
  Let $b$ be the point on $bd(U)$ such that $b<_{x_{0},U}x_{0}$ and
  for each point $y\in bd(U)$ with $y<_{x_{0},U}x_{0}$, $y\leq_{x_{0},U}b$. 

  Let $A_{1}$ and $A_{2}$ be a clopen partition of $M_{t}\backslash\{a\}$
  such that $b\in A_{1}$ and let $B_{1}$ and $B_{2}$ be a clopen partition
  of $M_{t}\backslash\{b\}$ such that $a\in B_{1}$. Furthermore, let
  $X_{+}$ and $X_{-}$ be a clopen partition of $M_{t}\backslash\{x_{0}\}$
  such that $a\in X_{+}$ and $b\in X_{-}$.

  First we prove that $x_{0}\in I(a,b)=A_{1}\cap B_{1}$. Assume for
  contradiction that $x_{0}\in A_{2}$. Since $a\in X_{+}$ it follows
  from lemma \ref{lem:disjoint-components} that $X_{-}\subset A_{2}$
  which implies that $b\in A_{2}$ which is a contradiction. Similarly,
  $x_{0}\in B_{1}$. Together this shows that $x_{0}\in A_{1}\cap B_{1}$.

  Next we'll prove that for every point $y\in I(a,b)\cap\overline{U}$
  we have $a\nsim_{y}b$ and $b<_{x_{0},U}y<_{x_{0},U}a$. Let $y$
  be some point in $I(a,b)\cap\overline{U}$. Without loss of generality,
  $y\in X_{+}$. Since in addition $y\in A_{1}$, $a\nless_{x_{0},U}y$
  which means that $y<_{x_{0},U}a$. By the definition of the order
  this implies that $x_{0}\nsim_{y}a$. Let $Y_{1}$ and $Y_{2}$ be
  a clopen partition of $M_{t}\backslash\{y\}$ such that $x_{0}\in Y_{1}$
  and $a\in Y_{2}$. By lemma \ref{lem:disjoint-components}, $X_{-}\subset Y_{1}$
  which means that $b\in Y_{1}$. This proves that $y<_{x_{0},U}a$
  and $a\nsim_{y}b$. Similarly, $b<_{x_{0},U}y$.

  We'll now see that $I(a,b)\subset U$. By lemma \ref{lem:interval-lemma},
  $I(x,y)\cap D$ is a non-empty open connected set. Furthermore, as
  we showed above, $x_{0}\in I(a,b)$. Assume for contradiction that
  $I(a,b)\backslash U\neq\emptyset$. Then by the connectedness of $I(a,b)\cap D$,
  $I(a,b)\cap bd(U)\neq\emptyset$. Let $y$ be a point in $I(a,b)\cap bd(U)$.
  Then as we saw before, $b<_{x_{0},U}y<_{x_{0},U}a$ which is clearly
  a contradiction to the choice of $a$ and $b$.\end{proof}
\begin{lem}
  Let $M$ be a 1-dimensional connected $\omega$-saturated t.t.t structure
  such that for every point $x\in M_{t}$, $d_{M}(M_{t}\backslash\{x\})>1$.
  Let $D\subset M_{t}$ be a connected open definable subset which is
  locally flat. Let $x_{0}$, $a$ and $b$ be points in $D$ such that
  $x_{0}\in I(a,b)\subset D$, for every point $y\in I(a,b)$ we have
  $a\nsim_{y}b$ and $<_{x_{0},I(a,b)}$ defines a dense linear order
  on $I(a,b)$. 

  Then for each pair of points $c,d\in I(a,b)$ such that $c<_{x_{0},I(a,b)}d$,
  \label{lem:subinterval-in-interval}\[
  I(c,d)=\{y\in I(a,b)\vert c<_{x_{0},I(a,b)}y<_{x_{0},I(a,b)}d\}\]
  .\end{lem}
\begin{proof}
  Let $C_{a}$ and $C_{b}$ be a clopen partition of $M_{t}\backslash\{c\}$
  such that $a\in C_{a}$ and $b\in C_{b}$ and let $D_{a}$ and $D_{b}$
  be a clopen partition of $M_{t}\backslash\{d\}$ such that $a\in D_{a}$
  and $b\in D_{b}$. Furthermore, let $A_{1}$ and $A_{2}$ be a clopen
  partition of $M_{t}\backslash\{a\}$ such that $c,d,b\in A_{1}$ and
  let $B_{1}$ and $B_{2}$ be a clopen partition of $M_{t}\backslash\{b\}$
  such that $c,d,a\in B_{1}$. 

  Let $X_{+}$ and $X_{-}$ be a clopen partition of $M_{t}\backslash\{x_{0}\}$
  such that $a\in X_{-}$ and $b\in X_{+}$. We'll assume that $c,d\in X_{+}$
  since the other cases are either similar or trivial. 

  Now we'll show that $x_{0}\in C_{a}$ and $x_{0}\in D_{a}$. Assume
  for contradiction that $x_{0}\in C_{b}$. Since $c\in X_{+}$, it
  follows from lemma \ref{lem:disjoint-components} that $X_{-}\cap C_{a}=\emptyset$
  which is a contradiction to the fact that $a\in X_{-}\cap C_{a}$.
  The proof that $x_{0}\in D_{a}$ is similar.

  Since $c<_{x_{0},I(a,b)}d$ and $x_{0}\in C_{a}\cap D_{a}$, by the
  definition of $<_{x_{0},I(a,b)}$ it follows that $d\in C_{b}$ and
  $c\in D_{a}$.

  We're now ready to prove the lemma.

  First we note that by lemma \ref{lem:disjoint-components}, it follows
  from the assumptions above that $I(a,b)=A_{1}\cap B_{1}$, $C_{b}\subset A_{1}$
  and $D_{a}\subset B_{1}$. In addition, since $d\in C_{b}$ and $c\in D_{a}$
  it follows from lemma \ref{lem:disjoint-components} that $I(c,d)=C_{b}\cap D_{a}$.
  Together this means that $I(c,d)\subset I(a,b)$.

  We'll now prove that \[
  I(c,d)=\{y\in I(a,b)\vert c<_{x_{0},I(a,b)}y<_{x_{0},I(a,b)}d\}\]
  Let $y$ be some point in $I(a,b)$ such that $c<_{x_{0},I(a,b)}y<_{x_{0},I(a,b)}d$.
  Since $c<_{x_{0},I(a,b)}y$, $y\in X_{+}$. In addition, $y\in X_{+}\cap C_{a}\Rightarrow c\nless_{x_{0},I(a,b)}y$
  so $y\in C_{b}$. In a similar fashion it follows that $y\in D_{a}$.
  Together this means that $y\in I(c,d)$.

  Now, let $y$ be a point in $I(c,d)=C_{b}\cap D_{a}$. Since $C_{b}\subset X_{+}$,
  $y\in X_{+}$. Furthermore, since $y\in C_{b}$, $c<_{x_{0},I(a,b)}y$.
  Finally, since $y\in D_{a}$, $d\nless_{x_{0},I(a,b)}y$ which means
  that $y<_{x_{0},I(a,b)}d$.
\end{proof}
Now, let's assume that $M$ and the set $D\subset M_{t}$ fulfill
the assumptions in lemma \ref{lem:exists-interval-with-order}. Then
for each point $x_{0}\in D$ there exists a pair of points $a$ and
$b$ in $D$ such that $x_{0}\in I(a,b)\subset D$, for every point
$y\in I(a,b)$ we have $a\nsim_{y}b$ and $<_{x_{0},I(a,b)}$ defines
a dense linear order on $I(a,b)$. 

By lemma \ref{lem:subinterval-in-interval}, this means that for all
$c,d\in I(a,b)$, $I(c,d)\subset I(a,b)$ and: \[
\{y\in I(a,b)\vert c<_{x,V_{x}}y<_{x,V_{x}}d\}=I(c,d)\]
. In other words, in the set $I(a,b)$ guaranteed by lemma \ref{lem:exists-interval-with-order},
the notion of an interval we defined above coincides with the interval
induced by the order $<_{x_{0},I(a,b)}$.

We'll now prove three lemmas about locally flat points which together
will show that the order we defined above can be extended from locally
flat points to connected locally flat sets.
\begin{lem}
  Let $M$ be a 1-dimensional connected $\omega$-saturated t.t.t structure
  such that for every element $x\in M_{t}$, $d_{M}(M_{t}\backslash\{x\})>1$.
  Let $D\subset M_{t}$ be a connected open definable subset which is
  locally flat. Let $x$, $a$ and $b$ be points in $D$ such that
  for every open set $U\subset D$ containing $x$ there exists some
  point $u\in U$ such that $a\nsim_{u}b$. Then $a\nsim_{x}b$.\label{lem:locally-flat-closed-sep}\end{lem}
\begin{proof}
  Since $M_{t}$ is Hausdorff, there exists an open definable connected
  set $U\subset D\backslash\{a,b\}$ containing $x$. By lemma \ref{lem:exists-interval-with-order},
  there exists a pair of points $c$ and $d$ in $U$ such that $x\in I(c,d)\subset U$,
  for every point $y\in I(c,d)$ we have $c\nsim_{y}d$ and $<_{x,I(c,d)}$
  defines a dense linear order on $I(c,d)$. 

  By the assumptions of the lemma, there exists some $y\in I(c,d)$
  such that $a\nsim_{y}b$. By lemma \ref{lem:subinterval-in-interval},
  $I(x,y)\subset I(c,d)$ which means that $a,b\notin I(x,y)$. So by
  lemma \ref{lem:interval-lemma}, $a\nsim_{x}b$.\end{proof}
\begin{lem}
  Let $M$ be a 1-dimensional connected $\omega$-saturated t.t.t structure
  such that for all $x\in M_{t}$, $d_{M}(M_{t}\backslash\{x\})>1$.
  Let $D\subset M_{t}$ be a connected open definable subset which is
  locally flat. Let $x$, $a$ and $b$ be points in $D$ such that
  $a\nsim_{x}b$. The there exists a definable open set $U\subset D$
  containing $x$ such that for every $u\in U$, $a\nsim_{u}b$.\label{lem:locally-flat-open-sep}\end{lem}
\begin{proof}
  As in the previous lemma, there exists an open definable connected
  set $U\subset D\backslash\{a,b\}$ containing $x$. By lemma \ref{lem:exists-interval-with-order},
  there exists a pair of points $c$ and $d$ in $U$ such that $x\in I(c,d)\subset U$,
  for every point $y\in I(c,d)$ we have $c\nsim_{y}d$ and $<_{x,I(c,d)}$
  defines a dense linear order on $I(c,d)$. 

  Let's choose some point $y\in I(c,d)$. Since $a,b\notin I(c,d)$ it 
  follows that $a,b\notin I(x,y)\subset I(c,d)$.
  Therefore, by lemma \ref{lem:interval-lemma}, $a\nsim_{y}b$.
\end{proof}
We now use the previous two lemmas to show that in some well defined
sense, locally flat sets look like a line.
\begin{lem}
  Let $M$ be a 1-dimensional $\omega$-saturated t.t.t structure such
  that for every point $x\in M_{t}$, $d_{M}(M_{t}\backslash\{x\})>1$.
  Let $D\subset M_{t}$ be a connected open definable subset which is
  locally flat. Then, there doesn't exist a definable connected closed
  subset $F\subset D$ such that $bd(F)>2$.\label{lem:locally-flat-connected-boundary-not-bigger-2}\end{lem}
\begin{proof}
  Assume for contradiction that $F\subset D$ is a definable closed
  connected subset and that $a,b,c\in bd(F)$.

  Let $F_{ab}$ denote the set of points $x\in F$ such that $a\nsim_{x}b$.
  \begin{claim*}
    $F_{ab}\neq\emptyset$.\end{claim*}
  \begin{proof}
    By lemma \ref{lem:exists-interval-with-order}, there exists a pair
    of points $x$ and $y$ in $D$ such that $a\in I(x,y)\subset D$,
    for every point $z\in I(x,y)$ we have $x\nsim_{z}y$ and $<_{a,I(x,y)}$
    defines a dense linear order on $I(x,y)$. 

    Let $s$ and $t$ be points in $I(x,y)$ such that $s<_{a,I(x,y)}a<_{a,I(x,y)}t$.
    By lemma \ref{lem:subinterval-in-interval}, \[
    \{z\in I(a,b)\vert s<_{a,I(x,y)}z<_{a,I(x,y)}t\}=I(s,t)\]
    so we can use the notion $I(s,t)$ to represent the interval given
    by the order $<_{a,I(x,y)}$. We'll use the result throughout the
    proof of the claim.

    Since $I(s,t)=I(s,a)\cup\{a\}\cup I(a,t)$ and $a\in bd(U)$, we can
    assume without a loss of generality that for every point $v\in I(a,t)$
    there exists a point $u\in I(a,v)$ such that $u\in F$. Therefore,
    there exists some point $u\in I(a,t)$ such that $I(a,u)\subset F$.
    Because assume for contradiction that for every $u\in I(a,t)$ there
    existed some $v\in I(a,u)$ such that $v\notin F$. In that case,
    the definable set $I(a,t)\cap F$ would have an infinite number of
    connected components which is a contradiction to $M$ being t.t.t.

    Now, since $a\notin int(F)$, for every point $u\in I(s,a)$ there
    must be some point $v\in I(u,a)$ such that $v\notin F$. Similarly
    to above, this means that there exists some point $v\in I(s,a)$ such
    that $I(v,a)\cap F=\emptyset$.

    Together, this means that without loss of generality we can assume
    that $I(x,a)\cap F=\emptyset$ and $I(a,y)\subset F$.

    If $b\in I(a,y)$ then for any $u\in I(a,b)\subset I(a,y)\subset F$,
    $a\nsim_{u}b$ and so $u\in F_{ab}$.

    Let's assume that $b\notin I(a,y)$. 

    Let $A_{1}$ and $A_{2}$ be a clopen partition of $M_{t}\backslash\{a\}$
    such that $x\in A_{1}$ and let $X_{1}$ and $X_{2}$ be a clopen
    partition of $M_{t}\backslash\{x\}$ such that $a\in X_{1}$. By the
    choice of $A_{1}$ and $X_{1}$, $I(x,a)=A_{1}\cap X_{1}$. Furthermore,
    since $I(x,a)\cap F=\emptyset$ and $U$ is connected, from lemma
    \ref{lem:disjoint-components} it follows that $F\subset\overline{A_{2}}$
    or $F\subset\overline{X_{2}}$. But $I(a,y)\subset A_{2}$ and $I(a,y)\subset F$
    so it must be that $F\subset\overline{A_{2}}=A_{2}\cup\{a\}$. This
    means that $b\in A_{2}$.

    Now, let $u$ be some point in $I(a,y)\subset F$. Let $U_{1}$ and
    $U_{2}$ be a clopen partition of $M_{t}\backslash\{u\}$ such that
    $a\in U_{1}$. Assume for contradiction that $b\in U_{1}$. Since
    $b\in A_{2}$ this means that $b\in A_{2}\cap U_{1}$. But $a\in U_{1}$
    and $u\in A_{2}$ which means that $I(a,u)=A_{2}\cap U_{1}$. This
    implies that $b\in I(a,u)\subset I(a,y)$ which is a contradiction
    to our assumption. Therefore, $b\in U_{1}$ and $a\in U_{2}$ which
    means that $u\in F_{ab}$.

    This concludes the proof of the claim.
  \end{proof}

  By lemmas \ref{lem:locally-flat-open-sep} and \ref{lem:locally-flat-closed-sep},
  $U_{ab}$ is clopen. Therefore, $U=U_{ab}$. 

  Similarly, if $U_{ac}$ and $U_{bc}$ are defined in the analogous
  fashion, $U=U_{ac}=U_{bc}=U_{ab}$. We'll now show that this is a
  contradiction.

  Let's choose a point $x\in int(U)$. Let $X_{1}$ and $X_{2}$ be
  the two connected components of $D\backslash\{x\}$. Either $X_{1}$
  or $X_{2}$ will contain two out of $a$, $b$, and $c$. Without
  loss of generality, $a,b\in X_{1}$. However, since $x\in U_{ab}$,
  $a\nsim_{x}b$ which is clearly a contradiction.
\end{proof}

We're now ready to show that every locally flat set is o-minimal.
In order to do this, we'll extend our previous notion of order from
neighborhoods of locally flat points to locally flat sets.

Let $D$ be a definable open connected locally flat set. Let $a\in D$
be some arbitrary point which we'll think of as the center. In addition,
let $D_{+}$ and $D_{-}$ be the two connected components of $D\backslash\{a\}$
which we'll think of as the {}``positive side'' and the {}``negative
side''. Finally let $x,y\in D$. We say that $x<_{a,D}y$ if one
of the following holds:
\begin{itemize}
\item $x,y\in D_{+}$ and $a\nsim_{x}y$
\item $x,y\in D_{-}$ and $a\nsim_{y}x$
\item $y\in D_{+}$ and $x\in D_{-}$
\item $y=a$ and $x\in D_{-}$
\item $x=a$ and $y\in D_{+}$
\end{itemize}
By proposition \ref{prop:definable-relation}, $<_{a,D}$ is definable. 

The next proposition shows that $<_{a,D}$ defines a dense linear
order on $D$ such that the induced interval topology is equivalent
to the topology induced by $M_{t}$.
\begin{prop}
  Let $M$ be a 1-dimensional connected $\omega$-saturated t.t.t structure
  such that for each point $x\in M_{t}$, $d_{M}(M_{t}\backslash\{x\})>1$.
  Let $D\subset M_{t}$ be a connected open definable subset which is
  locally flat and let $a$ be some point in $D$. Then $<_{a,D}$ defines
  a dense linear order on $D$ such that the induced interval topology
  is equivalent to the topology induced by $M_{t}$.\label{prop:locally-flat-set-o-minimal}\end{prop}
\begin{proof}
  Let $D_{+}$, $D_{-}$ be the sets used in the definition of $<_{a,D}$
  above. 

  Let $x,y,z\in D$. In addition, let $X_{1}$ and $X_{2}$ be the connected
  components of $D\backslash\{x\}$ such that $a\in X_{1}$. $Y_{1}$,
  $Y_{2}$, $Z_{1}$ and $Z_{2}$ are defined analogously for $y$ and
  $z$.
  \begin{enumerate}
  \item $x\nless_{a,D}y\Rightarrow y<_{a,D}x$:

    We assume that $x,y\in D_{+}$. The other possibilities are either
    identical or trivial. By the assumption, $y\in X_{1}$. Assume for
    contradiction that $x\in Y_{1}$. Since $a\in X_{1}\cap Y_{1}$, by
    lemma \ref{lem:disjoint-components} both $X_{2}$ and $Y_{2}$ are
    subsets of $D_{+}$. 

    Let's define \[
    F=(D_{+}\cap X_{1}\cap Y_{1})\cup\{a\}\cup\{x\}\cup\{y\}\]
    . By lemma \ref{lem:disjoint-components} $F$ is closed and $bd(F)\subset\{a,x,y\}$.
    We'll now show that $F$ is connected and that $bd(F)=\{a,x,y\}$. 

    Let $C$ be some connected component of $F$. Since $D$ is connected,
    $bd(C)=\{a,x,y\}$. For assume for contradiction that one of the points
    in the set $\{a,x,y\}$ was not included in $bd(C)$. Without loss
    of generality, let's assume that $bd(C)=\{x,y\}$. Since $bd(C)\subset\{a,x,y\}$,
    $bd(X_{2})=\{x\}$ and $bd(Y_{2})=\{y\}$, it follows that $C\cup X_{2}\cup X_{1}\cup\{x\}\cup\{y\}$
    is a clopen subset of $D$ which is clearly a contradiction. Assuming
    that $bd(C)$ is equal to some other strict subset of $\{a,x,y\}$
    gives a similar contradiction.

    In addition, $a$, $x$ and $y$ are boundary points of $D_{-}$,
    $X_{2}$ and $Y_{2}$ respectively so by lemma \ref{lem:flat-two-components},
    each one of $a$, $x$, and $y$ is the boundary point of at most
    one connected component of $F$. Therefore, $F$ has only one connected
    component and $bd(F)=\{a,x,y\}$.

    However, this is a contradiction to lemma \ref{lem:locally-flat-connected-boundary-not-bigger-2}.

  \item $x<_{a,D}y\Rightarrow y\nless_{a,D}x$:

    Also in this case we'll assume that $x,y\in D_{+}$. Since $x<_{a,D}y$
    and $a\in X_{1}$, we get that $y\in X_{2}$. Now, we claim that $x\in Y_{1}$.
    For otherwise, if it was true that $x\in Y_{2}$ then we'd get from
    lemma \ref{lem:disjoint-components} that $Y_{1}\cap X_{1}=\emptyset$
    which is a contradiction to the fact that $a\in X_{1}\cap Y_{1}$.

  \item $x<_{a,D}y\wedge y<_{a,D}z\Rightarrow x<_{a,D}z$:

    According to the assumptions, $a,x\in Y_{1}$, $z\in Y_{2}$, $a\in X_{1}$
    and $y\in X_{2}$. We have to prove that $z\in X_{2}$ as well. But
    again by lemma \ref{lem:disjoint-components}, \[
    y\in X_{2}\wedge x\in Y_{1}\Rightarrow Y_{2}\subset X_{2}\Rightarrow z\in X_{2}\]
    .

  \end{enumerate}
  This shows that $<_{a,D}$ is a linear order. We'll now show that
  $<_{a,D}$ is dense. 

  Let's assume that $x<_{a,D}y$. As we showed above, this means that
  $y\in X_{2}$ and $x\in Y_{1}$. By lemma \ref{lem:interval-lemma},
  $X_{2}\cap Y_{1}=I(x,y)\neq\emptyset$. Let $s$ be some point in
  $I(x,y)$. Since $s\in X_{2}\cap Y_{1}$, it follows from the definition
  of $<_{a,D}$ that $x<_{a,D}s<_{a,D}y$.

  We'll now see that the order topology induced on $D$ by $<_{a,D}$
  is equivalent to the topology on $D$ induced by $M_{t}$.

  As a first step, we note that if $x<_{a,D}y$, then \[
  I(x,y)\cap D=\{z\in D\vert x<_{a,D}z<_{a,D}y\}\]
  . This is immediate from the definitions of $I(x,y)$ and $<_{a,D}$.

  Let $U\subset D$ be an open set in $D$ with $x\in U$. By lemma
  \ref{lem:exists-interval-with-order}, there exists a pair of points
  $s,t\in U$ such that $x\in I(s,t)$ and $I(s,t)\subset U$. Without
  loss of generality, $s<_{a,D}t$.

  Therefore, \[
  x\in\{u\in D\vert s<_{a,D}u<_{a,D}t\}=I(s,t)\cap D\subset U\]
  . The other direction is trivial as $I(x,y)$ is open in $D$ for
  every $x,y\in D$.
\end{proof}
We now obtain our primary result as an immediate consequence of propositions
\ref{prop:finite-not-flat} and \ref{prop:locally-flat-set-o-minimal}.
\begin{thm}
  Let $M$ be a 1-dimensional connected $\omega$-saturated t.t.t structure
  such that for every point $x\in M_{t}$, $d_{M}(M_{t}\backslash\{x\})>1$.
  Then there exists a finite set $X\subset M_{t}$ such that each of
  the finite number of connected components of $M_{t}\backslash X$
  are o-minimal.\label{thm:remove-finite-o-minimal}\end{thm}
\begin{proof}
  Let $X$ be the definable set of points in $M_{t}$ which aren't locally
  flat. By proposition \ref{prop:finite-not-flat}, $X$ is finite.
  Let $D$ be a connected component of $M_{t}\backslash X$. Since there're
  only a finite number of connected components, $D$ is a connected
  open definable subset which is locally flat. By proposition \ref{prop:locally-flat-set-o-minimal},
  there exists a definable dense linear order which induces the topology
  on $D$. By \cite[6.2]{P1}, this means that $D$ is o-minimal.\end{proof}
\begin{rem*}
  Note that even if $M_{t}$ isn't connected, we can obtain theorem
  \ref{thm:remove-finite-o-minimal} for any open connected definable subset $D\subset M_t$
  with the property that removing any point from $D$ splits $D$ into more than
  one connected component.\end{rem*}
\begin{example}
  Let's return to the structure $R_{int}=\langle\mathbb{R},I(x,y,z)\rangle$
  from example \ref{exa:interval-example}. By theorem \ref{thm:remove-finite-o-minimal},
  we should be able to recover the standard order $<$ on $\mathbb{R}$
  from $I$.

  Let $D=\mathbb{R}$ and let $a$ be some point from $\mathbb{R}$.
  In addition, let $x$ and $y$ be points in $\mathbb{R}$ such that
  $a<x<y$. By the construction of, $<_{a,D}$ it's clear that $a<_{a,D}x<_{a,D}y$.
  By checking the other possibilities for $x$ and $y$ in a similar
  fashion it's easy to see that $<_{a},D$ is equivalent to $<$.
\end{example}

\section{Structures Without Splitting\label{sec:Structures-Without-Splitting}}

In this section we look at structures where removing a point doesn't
split the structure into more than one connected component. One example
of such a structure is the unit circle. Our main goal in this section
will be to find alternative topological properties which ensure that
the structure is at least locally o-minimal as in the case of the
unit circle.
\begin{lem}
  Let $M$ be an $\omega$-saturated one dimensional t.t.t structure.
  Let $A$ be a definable open set, $f:A\rightarrow\mathcal{P}(M_{t})$
  a function such that $f(x)$ is finite for each point $x\in M_{t}$
  and $\Gamma$ the graph of $f$. Then for each point $x\in A$, the
  fiber $(\overline{\Gamma})_{x}$ is finite.\label{lem:finite-fiber}\end{lem}
\begin{proof}
  Let's assume for contradiction that there exists a point $x\in A$
  and a sequence of points $(y_{i})_{i<\omega}$ in $M_{t}$ such that
  for every $i<\omega$, every basis set $U\subset A$ containing $x$
  and every basis point $V$ containing $y_{i}$, there exists some
  $z\in U\backslash\{x\}$ such that $f(z)\in V$.\end{proof}
\begin{claim*}
  For every $i<\omega$ and every basis set $V$ containing $y_{i}$
  there exists a basis set $U\subset A$ containing $x$ such that for
  every basis set $W\subset U$ containing $x$ we have $f(bd(W))\cap V\neq\emptyset$.\end{claim*}
\begin{proof}
  Let's take some $i<\omega$. Let $V$ be some basis set containing
  $y_{i}$. Assume for contradiction that for every basis set $U\subset A$
  containing $x$ there exists some basis set $W\subset U$ containing
  $x$ such that $f(bd(W))\cap V=\emptyset$.

  We now define $X=f^{-1}(V)\cap A$. By the definition of $y_{i}$,
  for every basis set $W$ containing $x$, $W\cap(X\backslash\{x\})\neq\emptyset$.
  Therefore, by the assumption for contradiction there exists a descending
  sequence of basis sets $(W_{i})_{i<\omega}$ such that for all $i<\omega$:
  \begin{itemize}
  \item $x\in W_{i}\subset X$
  \item $\overline{W_{i+1}}\subsetneq W_{i}$
  \item $bd(W_{i})\cap X=\emptyset$\end{itemize}
  \begin{proof}
    By the last two properties, for every $i<\omega$ the set $W_{i}\backslash W_{i+1}$
    is clopen in $X$. But this means that $X$ can be partitioned into
    an infinite number of definable clopen subsets which is a contradiction.
  \end{proof}
  Now, by the $\omega$-saturation we can assume that there exists some
  $N<\omega$ such that for every $z\in A$, $\vert f(z)\vert<N$. Similarly,
  there exists some $B<\omega$ such that for every basis set $V$,
  $\vert bd(V)\vert<B$. Let $V_{1},\dots,V_{NB+1}$ be pairwise disjoint
  basis sets such that for every $1\leq i\leq NB+1$ we have $y_{i}\in V_{i}$.
  By the claim, there exists a basis set $U\subset A$ containing $x$
  such that for every basis set $W\subset U$ containing $x$ and every
  every $1\leq i\leq NB+1$ we have $f(bd(W))\cap V_{i}\neq\emptyset$.

  Let $W\subset U$ be some basis set. By the definitions of $N$ and
  $B$, $\vert f(bd(W))\vert\leq NB$ which is a contradiction to the
  fact that $V_{1},\dots,V_{NB+1}$ are pairwise disjoint.\end{proof}
\begin{prop}
  Let $M$ be an $\omega$-saturated one dimensional t.t.t structure
  such that there exist a definable continuous function $F:M_{t}^{2}\rightarrow M_{t}$
  and a point $a\in M_{t}$ such that for each $x\in M_{t}$, $F(x,x)=a$
  and $F(x,\cdot)$ is injective. Let $f:M_{t}\rightarrow\mathcal{P}(M_{t})$
  be a function such that for every $x\in M_{t}$, $\vert f(x)\vert<\infty$
  and $f(x)\neq x$. Let $\Gamma$ be the graph of $f$. Then for every
  basis set $U\subset M_{t}$, there exists a point $x\in U$ such that
  $(x,x)\notin\overline{\Gamma}$.\label{prop:F-function}\end{prop}
\begin{proof}
  Assume for contradiction that there exists some basis set $U\subset M_{t}$
  such that for every $x\in U$, $(x,x)\in\overline{\Gamma}$. We now
  define the function $g:U\rightarrow \mathcal{P}(M_{t})$ by\[
  g(x)=\{F(x,y)\vert y\in f(x)\}\]
  . 

  In addition, we define the function $h:M_{t}\rightarrow\mathcal{P}(U)$
  by \[
  h(y)=g^{-1}(y)\cap U\]
  . Let $\Gamma_{g}$ and $\Gamma_{h}$ be the graphs of $g$ and $h$
  respectively.

  By our assumption on $F$, $F(x,x)=a$ for each $x\in U$. Therefore,
  by the continuity of $F$ together with the assumption that for every
  $x\in U$, $(x,x)\in\overline{\Gamma}$, we get that $(x,a)\in\overline{\Gamma_{g}}$
  for each $x\in U$.

  Furthermore, since for every $x\in U$ we have $f(x)\neq x$ and $F(x,\cdot)$
  is injective, $g(x)\neq a$ for all $x\in U$.

  We'll now show that there exists an open set $A$ containing $a$
  such that for every $y\in A$, the set $h(y)$ is finite. By our assumptions
  on $f$, $g(x)$ is finite for every $x\in M_{t}$. Therefore, by
  the exchange principle there are a finite number of points $y\in M_{t}$
  such that $h(y)$ is infinite. Furthermore, $h(a)=\emptyset$ so by
  the Hausdorffness of $M_{t}$, there exists an open set $A$ containing
  $a$ such that for every $y\in A$, the set $h(y)$ is finite.

  In addition, since $(x,a)\in\overline{\Gamma_{g}}$ for each $x\in U$,
  the fiber $(\overline{\Gamma_{h}})_{a}$ is infinite. However, this
  is a contradiction to lemma \ref{lem:finite-fiber}.\end{proof}

\begin{prop}
  Let $M$ be an $\omega$-saturated one dimensional t.t.t structure
  such that there exist a definable continuous function $F:M_{t}^{2}\rightarrow M_{t}$
  and a point $a\in M_{t}$ such that for each $x\in M_{t}$, $F(x,x)=a$
  and $F(x,\cdot)$ is injective. 

  Then for every basis set $U$ there exists a basis set $V\subset U$
  such that for every point $x\in V$ there exists a basis set $W$
  such that $bd(W)\cap V=\{x\}$. \label{prop:boundary-only-x}\end{prop}
\begin{proof}
  First of all, without loss of generality we can assume that for every
  $x\in M_{t}$ there is some basis set $W$ such that $x\in bd(W)$.
  Because let $X$ be the set of all such points. $X$ is clearly definable.
  Assume for contradiction that $X^{c}$ has a non empty interior. Let
  $W$ be a basis set such that $\overline{W}\subset X$. Then $bd(W)\subset X$
  which is clearly a contradiction. Therefore, $X^{c}$ is finite.

  We define a function $f:U\rightarrow\mathcal{P}(M_{t})$ by\[
  f(x)=\{y\neq x\vert\textrm{there exists a basis set }W\textrm{ such that }\{x,y\}\subset bd(W)\}\]

  Let $N$ be an integer such that for every basis set $U$, $\vert bd(U)\vert<N$.
  Let $\Gamma$ be the graph of $f$.

  We now look at two cases.

  For the first case let's assume that there exists a basis set $V\subset U$
  such that for each $x\in V$, $\vert f(x)\vert<\infty$. By proposition
  \ref{prop:F-function}, there exists some basis set $W\subset V$
  such that $(W\times W)\cap\Gamma=\emptyset$. This means that for every
  point $x\in W$, there exists a basis set $A$ such that $bd(A)\cap W=\{x\}$.
  Because let $x$ be some point in $W$ and let $A$ be a basis set
  such that $x\in bd(A)$. Since $(W\times W)\cap\Gamma=\emptyset$, the rest
  of the boundary points of $A$ are not contained in $W$ which means
  that $bd(A)\cap W=\{x\}$.

  For the second case, assume that for every basis set $V\subset U$
  there exists some point $x\in V$ such that $f(x)$ is infinite. We
  now assume for contradiction that there doesn't exist a basis set
  $V\subset U$ such that for every point $x\in V$, there exists a
  basis set $W$ such that $bd(W)\cap V=\{x\}$. 

  In order to get a contradiction, we'll inductively build a sequence
  of tuples of points, basis sets and functions $(x_{i},V_{i},f_{i})_{i=1}^{N}$
  with the following properties:
  \begin{itemize}
  \item $V_{1}$ is an arbitrary basis set in $U$, $x_{1}$ is a point in
    $V_{1}$ such that $f(x_{1})$ is infinite and $f_{1}=f$.
  \item For all $1\leq i\leq N$, $f_{i}:U\rightarrow\mathcal{P}(M_{t})$
    is defined by \[
    f_{i}(x)=\{y\neq x,x_{1},\dots,x_{i-1}\vert\textrm{there exists a basis set }W\]
    \[
    \textrm{ such that }\{x,x_{1},\dots,x_{i-1},y\}\subset bd(W)\}\]

  \item For all $1\leq i\leq N$, $x_{i}\in V_{i}$
  \item For all $1\leq i\leq N$, $f_{i}(x_{i})$ is infinite.
  \item For all $i<j$, $x_{i}\notin V_{j}$
  \end{itemize}
  The existence of $(x_{1},V_{1},f_{1})$ follows immediately from our
  assumptions in the second case.

  Let's assume that we've constructed the sequence up to the $i$-th
  place. Since $f_{i}(x_{i})$ is infinite, there exists some basis
  set $V_{i+1}\subset f_{i}(x_{i})$ such that $x_{i}\notin V_{i+1}$.
  We define $f_{i+1}$ as above.

  Now, if for all $x\in V_{i+1}$ the set $f_{i+1}(x)$ would be finite
  then just as in the first case, together with the fact that $i<j\Rightarrow x_{i}\notin V_{j}$,
  there would exist a basis set $W\subset U$ such that for every $y\in W$,
  there exists a basis set $A$ such that $bd(A)\cap W=\{y\}$. Therefore,
  by our assumption for contradiction, there exists some point $x_{i+1}\in V_{i+1}$
  such that $f_{i+1}(x_{i+1})$ is infinite. Thus we've found a tuple
  $(x_{i+1},V_{i+1},f_{i+1})$ satisfying the requirements.

  However, the existence of the tuple $(x_{N},V_{N},f_{N})$ is clearly
  a contradiction because on the one hand $f_{N}(x_{N})$ is infinite
  but by the definition of $N$, for every point $x\in U$ the set $f_{N}(x)$
  is empty.
\end{proof}
We'll now prove a similar proposition under the assumption that all
of the basis sets have only two points in their boundary.
\begin{lem}
  Let $X$ be a topological space and let $U\subset X$ and $V\subset X$
  be connected open sets such that \[
  bd(U)\cap V=bd(V)\cap U=\emptyset\]
  and $U\neq V$. Then $U\cap V=\emptyset$.\label{lem:connected-same-boundary}\end{lem}
\begin{proof}
  Let's look at the open set $W=U\cap V$. If $W=\emptyset$ then we're
  finished.

  Let's assume that $W\neq\emptyset$. If $W\neq U$ then since $U$
  is connected, the boundary of $W$ in $U$ must be non-empty. Let $x\in U$
  be a point in $bd(W)$. Since $x\notin W$, $x\notin V$. But $x\in\overline{W}\subset\overline{V}$
  which means that $x\in bd(V)$. This is a contradiction to the fact
  that $x\in U$. Therefore, $W=U$. Similarly, $W=V$. Together this
  means that $U=V$ which is a contradiction to the assumption. \end{proof}
\begin{lem}
  Let $M$ be a t.t.t structure and let $X\subset M_{t}$ be some finite
  subset. Then there are only a finite number of basis sets $U\subset M_{t}$
  such that $bd(U)=X$.\label{lem:finite-basis-sets}\end{lem}
\begin{proof}
  Let $\mathcal{B}$ be the set of basis sets $U$ such that $bd(U)=X$.
  Assume for contradiction that $\mathcal{B}$ is infinite. Let $\mathcal{C}$
  be defined by\[
  \mathcal{C}=\bigcup_{U\in\mathcal{B}}\{C\subset M_{t}\vert C\textrm{ is a connected component of }U\}\]
  For each $C\in\mathcal{C}$, $bd(C)\subset X$. In addition, since
  $\mathcal{B}$ is infinite, $\mathcal{C}$ is infinite as well. However,
  by lemma \ref{lem:connected-same-boundary}, for each pair of connected
  components $C_{1},C_{2}\in\mathcal{C}$, $C_{1}\cap C_{2}=\emptyset$.
  Therefore, the definable set \[
  \bigcup\mathcal{C}=\bigcup\mathcal{B}\]
  can be partitioned into an infinite number of clopen sets which is
  a contradiction to the fact that $M$ is t.t.t.\end{proof}
\begin{prop}
  Let $M$ be an $\omega$-saturated one dimensional t.t.t structure
  such that for every basis set $U$, $\vert bd(U)\vert=2$.

  Then for every basis set $U$ there exists a basis set $V\subset U$
  such that for every point $x\in V$ there exists a basis set $W$
  such that $bd(W)\cap V=\{x\}$.\label{pro:every-basis-boundary-2}\end{prop}
\begin{proof}
  As before, without loss of generality we can assume that for every
  $x\in M_{t}$ there is some basis set $W$ such that $x\in bd(W)$. 

  We also use the function $f:U\rightarrow\mathcal{P}(M_{t})$ defined
  above by\[
  f(x)=\{y\neq x\vert\textrm{there exists a basis set }W\textrm{ such that }\{x,y\}\subset bd(W)\}\]

  Let $U$ be some basis set. First let's assume that there exists a
  point $u\in U$ such that $\vert f(u)\cap U\vert=\infty$. Since $M_{t}$ is
  Hausdorff, there exists some basis set $V\subset f(u)\cap U$ such
  that $u\notin V$. $V$ clearly satisfies the requirements of the
  proposition. 

  On the other hand, assume that for each point $u\in U$, $\vert f(u)\cap U\vert<\infty$.
  By lemma \ref{lem:finite-basis-sets} this means that there are only
  a finite number of basis sets $W\subset U$ such that $u\in bd(W)$.
  By the $\omega$-saturation this means that there exists some number
  $N\in\mathbb{N}$ such that for each point $u\in U$ there are at
  most $N$ basis sets $W\subset U$ such that $x\in bd(W)$. 

  We'll now show using downward induction that there exists a basis
  set $V\subset U$ such that for every point $v\in V$, there are no
  basis sets $W\subset V$ such that $v\in bd(W)$ which is clearly
  a contradiction.

  Assume that we found a basis set $V_{i}$ for $0<i\leq N$ such that
  for every point $v\in V_{i}$, there are at most $i$ basis sets $W\subset U$
  such that $x\in bd(W)$. Let $v_{i}$ be some point in $V_{i}$ and
  let $X$ be the set of points $x\in V_{i}$ such that there exists
  a basis set $W$ with $v_{i}\in W$ and $x\in bd(W)$. Again by the
  fact that $M_{t}$ is Hausdorff it follows that $X$ is infinite.
  We choose $V_{i-1}$ to be some basis set such that $V_{i-1}\subset X$
  and $v_{i}\notin V_{i-1}$.

  Now, let $x$ be some point in $V_{i-1}$. Since $x\in V_{i}$, there
  are at most $i$ basis sets $W\subset V_{i}$ such that $x\in bd(W)$.
  However, one of these sets contains $v_{i}$ which isn't an element
  in $V_{i-1}$. Therefore, there are at most $i-1$ basis sets $W\subset V_{i-1}$
  such that $x\in bd(W)$. 

  This finishes the downward induction and the proposition.\end{proof}
\begin{thm}
  Let $M$ be a 1-dimensional $\omega$-saturated t.t.t structure
  such that one of the following holds:
  \begin{enumerate}
  \item There exist a definable continuous function $F:M_{t}^{2}\rightarrow M_{t}$
    and a point $a\in M_{t}$ such that for each $x\in M_{t}$, $F(x,x)=a$
    and $F(x,\cdot)$ is injective.
  \item For every basis set $U$, $\vert bd(U)\vert=2$.
  \end{enumerate}
  Then for all but a finite number of points, for every point $x\in M_{t}$
  there's a basis set $U$ containing $x$ such that $U$ is o-minimal.\label{thm:o-min-neighborhood}\end{thm}
\begin{proof}
  It's enough to prove that for every basis set $U$ there
  exists a point $x_{0}\in U$ with an o-minimal neighborhood.

  Let $U$ be a basis set. By propositions \ref{prop:boundary-only-x}
  and \ref{pro:every-basis-boundary-2}, there exists a basis set $V\subset U$
  such that for every point $x\in V$ there exists a basis set $W$
  such that $bd(W)\cap V=\{x\}$. This means that for every point $x\in V$,
  $V\backslash\{x\}$ has at least two connected components. Without loss 
  of generality $V$ is connected. By theorem
  \ref{thm:remove-finite-o-minimal} (and the remark immediately after it), 
  this means that after removing
  a finite number of points from $V$ the remaining connected components
  are o-minimal. Let $C$ be one of the o-minimal components and let
  $x_{0}$ be some point in $C$. Clearly $x_{0}$ has an o-minimal
  neighborhood.\end{proof}
\begin{cor}
  Let $M$ be an $\omega$-saturated one dimensional t.t.t structure
  which  admits a topological group structure . Then all but a finite number of points
  have an o-minimal neighborhood.\end{cor}
\begin{proof}
  We define a function $F:M_{t}^{2}\rightarrow M_{t}$ by\[
  F(x,y)=x-y\]
  The function $F$ clearly satisfies the conditions of theorem \ref{thm:o-min-neighborhood}.
\end{proof}
\begin{prop}
  Let $M$ be an $\omega$-saturated one dimensional t.t.t structure
  such that all but a finite number of points have an o-minimal neighborhood.
  Let $f:M_{t}\rightarrow M_{t}$ be definable function. Then $f$ is
  continuous for all but a finite number of points.\end{prop}
\begin{proof}
  We'll show that in every basis set $U$ there is a point $x_{0}\in U$
  such that $f$ is continuous at $x_{0}$.

  Let $U$ be a basis set and let $\Gamma$ be the graph of $f$. 

  If the projection of $\Gamma\cap(U\times M_{t})$ onto the second coordinate
  is finite then there exists a basis set $V\subset U$ such that $f$ is constant on $V$.

  On the other hand, if the projection is infinite then there exists some point
  $y\in M_{t}$ with an o-minimal neighborhood such that for any basis set $V$ containing $y$ there
  exists some $x\in U$ such that $y\neq f(x)\in V$. Let $V$ be an o-minimal basis set containing $y$.
  By the choice of $y$, $f^{-1}(V)\cap U$ is infinite so by the assumption
  there exists some o-minimal basis set $W\subset U$ such that $f(W)\subset V$.
  Since both $W$ and $V$ are o-minimal, by the monotonicity theorem
  there exists some $x_{0}\in W\subset U$ such that $f$ is continuous
  at $x_{0}$.
\end{proof}
We conclude this section by giving two examples of theorem \ref{thm:o-min-neighborhood}.
\begin{example}
  For the first example we return to the unit circle mentioned in the
  beginning of the section. We'll look at the structure $S=\langle S^{1},R(x,y,z)\rangle$
  where $R(x,y,z)$ is true if $x$ and $y$ are not opposite each other
  and $z$ lies on the short arc between $x$ and $y$. If we define
  the set of basis sets as\[
  \{R^{S}(a,b,z)\vert a,b\in S\}\]
  then $S$ is a 1-dimensional $\omega$-saturated t.t.t structure. In addition,
  for every basis set $U$, $\vert bd(U)\vert=2$ so by theorem \ref{thm:o-min-neighborhood}
  $S$ is locally o-minimal. This is indeed true as locally $S$ looks
  like the structure $R_{int}$ from example \ref{exa:interval-example}
  which was shown to be o-minimal.
\end{example}

\begin{example}
  This example is a slight variant of $R_{int}=\langle\mathbb{R},I(x,y,z)\rangle$.
  Let's define the relation $RI(x,y,z)$ in $R=\langle\mathbb{R},+,\cdot,0,1,<\rangle$
  by:

  \[
  RI(x,y,z)\iff(x<z<y)\wedge(-1<x-y<1)\]
  So $RI$ is a version of $I$ restricted to intervals with a length
  of less than 1. 

  Let $R\prec\overline{R}$ be an $\omega$-saturated elementary extension
  and let $\overline{\mathbb{R}}$ be the universe of $\overline{R}$.
  We define $\overline{R}_{rint}=\langle\overline{\mathbb{R}},RI\rangle$
  to be the restriction of $\overline{R}$ to the language $\{RI\}$.

  Clearly $\overline{R}_{rint}$ is $\omega$-saturated. In addition,
  since $\overline{R}$ is o-minimal, $\overline{R}_{rint}$ is a one
  dimensional t.t.t structure. 

  However, for any point $a\in\overline{\mathbb{R}}$, $\overline{\mathbb{R}}\backslash\{a\}$
  has only one definably connected component in $R_{rint}$. Because
  assume for contradiction that the sets \[
  A_{+}=\{x\in\overline{\mathbb{R}}\vert x>a\}\]
  \[
  A_{-}=\{x\in\overline{\mathbb{R}}\vert x<a\}\]

  were definable in $\overline{R}_{rint}$ using the constants $c_{1},\dots,c_{N}$.
  Let's define subsets\[
  \tilde{A}_{+}=\{x\in\overline{\mathbb{R}}\vert(x>a)\wedge(\forall n<N\forall k(x>c_{n}+k)\}\]

  \[
  \tilde{A}_{-}=\{x\in\overline{\mathbb{R}}\vert(x<a)\wedge(\forall n<N\forall k(x<c_{n}+k)\}\]

  By the definition of $RI$, an automorphism of $\overline{\mathbb{R}}$ which
  swaps $\tilde{A}_{+}$ and $\tilde{A}_{-}$ is an automorphism of
  $\overline{R}_{rint}$ together with the constants $c_{1},\dots,c_{N}$
  which is a contradiction.

  On the other hand, the basis sets of $R_{rint}$ have two boundary
  points so by theorem \ref{thm:o-min-neighborhood}, every point in
  $R_{rint}$ has an o-minimal neighborhood. This makes sense because
  for any point $a\in\overline{\mathbb{R}}$ and interval $U$ containing
  $a$ with a length of less than one, we can define an order on $U$
  in the same way that we defined an order on $R_{int}$ in example
  \ref{exa:interval-example}. 
\end{example}

\subsection*{Acknowledgement}I'd like to thank Ehud Hrushovski for providing valuable insights and guidance throughout
the time spent working on this paper.

\bibliographystyle{amsplain}
\bibliography{ttt-arxiv}

\end{document}